\documentclass[12pt,a4paper]{article}
\usepackage{amsmath,amssymb,amsthm,graphicx,color}
\usepackage[left=1in,right=1in,top=1in,bottom=1in]{geometry}
\usepackage{cite}
\usepackage[british]{babel}

\newtheorem{theorem}{Theorem}[section]
\newtheorem{lemma}[theorem]{Lemma}

\newtheorem{conjecture}[theorem]{Conjecture}

\theoremstyle{definition}

\theoremstyle{remark}
\newtheorem{remark}[theorem]{Remark}

\allowdisplaybreaks

\numberwithin{equation}{section}

\newcommand{\R}{{\mathbb{R}}}
\newcommand{\RP}{{\mathbb{R}_+}}

\newcommand{\ZP}{{\mathbb{Z}_+}}
\newcommand{\N}{{\mathbb{N}}}
\newcommand\Exp{{\mathbb{E}}}
\renewcommand\Pr{{\mathbb{P}}}

\newcommand\1{{\mathbf 1}}
\newcommand\eps{\varepsilon}
\newcommand\re{{\mathrm{e}}}
\newcommand{\ud}{{\mathrm{d}}}

\newcommand{\cX}{{\mathcal{X}}}
\newcommand{\cH}{{\mathcal{H}}}

\newcommand{\cF}{{\mathcal{F}}}
\newcommand{\cG}{{\mathcal{G}}}
\newcommand{\cV}{{\mathcal{V}}}
\newcommand{\cW}{{\mathcal{W}}}

\newcommand{\as}{\ \textrm{a.s.}}
\newcommand{\argmin}{\mathop{{\rm arg} \min}}

\newcommand{\hier}[3][X]{#1^{(#2)}_{#3}}

\title{Phase transitions for random geometric
 preferential attachment graphs}
\author{Jonathan Jordan\\
\normalsize{University of Sheffield}\\
\and Andrew R. Wade\\
\normalsize{Durham University}}

\begin{document}

\maketitle

\begin{abstract}
We study an evolving
spatial network 
in which sequentially arriving vertices are joined
to existing vertices at random according to a rule that combines preference according to degree with preference
according to spatial proximity. We investigate phase transitions in graph structure as the relative
weighting of these two components of
 the attachment rule is varied. 

Previous work of one of the authors showed that when the
geometric component is weak, the 
limiting degree sequence of the resulting graph coincides with that of the standard Barab\'asi--Albert preferential attachment model.
We show that at the other extreme, in the case of a sufficiently strong geometric component, the
limiting degree sequence coincides with that of a purely
geometric model, the on-line nearest-neighbour graph,  which is of interest in its own right and for which we prove some extensions of known results. We also
show the presence of an intermediate regime, in which the behaviour differs significantly
from both the on-line nearest-neighbour graph and the Barab\'asi--Albert model; in this regime, we obtain
a stretched exponential upper bound on the degree sequence.

Our results lend some mathematical support to simulation studies of Manna and Sen,
while proving that the power law to stretched exponential phase transition occurs at a different point from the one
conjectured by those authors.
\end{abstract}



\section{Introduction}
\label{sec:intro}

Stochastic models for network evolution have been the subject of an explosion of interest over the past decade or so,
motivated by real-world graphs such as those associated with social networks or the internet: see e.g.\
\cite{bonato,durrett} for an introduction to some of the vast literature and some of the key models.
In a typical setting, a graph is grown via the sequential addition of new nodes, and each new node
is connected by an edge to an existing node in the graph according to some (often probabilistic) rule.
Several popular connectivity rules are based on 
 {\em preferential attachment}, whereby the random endpoint of the new edge is chosen
with probability proportional to the current vertex degrees:
the preferential attachment paradigm is supposed to capture the idea that in many real-world networks
highly-connected nodes are more likely to attract new connections.
On the other hand, real-world networks
often have spatial content, and so other network growth models
assign to each vertex a (random) spatial location and have a {\em geometric} connectivity rule;
the {\em on-line nearest-neighbour graph}, for example, is constructed by connecting each
new vertex to its nearest neighbour among its predecessors.

The subject of this paper is a model whose connectivity rule combines a degree of
preferential attachment with a spatial, distance-dependent component; we describe our model
 in detail below. This model, previously studied in \cite{jordan2010},
 is a variant of the geometric preferential attachment model
of Flaxman {\em et al.} \cite{ffv,ffv2}, which itself can be viewed as a generalization of an earlier model
of Manna and Sen \cite{ms}. A continuous time model with a similar flavour has recently been studied
by Jacob and M\"orters \cite{jm}.

In a sense that we will explain in this paper,
the   behaviour of the geometric preferential attachment model considered here
 interpolates between pure preferential attachment
(essentially the well-known Barab\'asi--Albert model)
and a purely geometric model (the on-line nearest-neighbour graph).
It was shown in \cite{jordan2010} that for a sufficiently weak geometric
component of the attachment rule, the limiting degree distribution
coincides with that of the Barab\'asi--Albert model,
which famously has a `scale-free' or `power-law' degree distribution \cite{brst,jordan2006}.

The focus of the present paper is the complementary setting, in which the geometric component 
has a significant impact.
We show that in the extreme case of a dominant geometric effect,
 the model behaves similarly to the on-line nearest-neighbour graph,
which by contrast has a degree distribution with exponential tails (cf \cite{bbbcr}).
We also study an intermediate regime in which the model
behaves differently from both of the extreme cases, and in which the degree
distribution satisfies a stretched exponential
tail bound.
Thus we demonstrate the existence of non-trivial phase transitions for the model.

\section{Random spatial graph models and main results}
\label{sec:results}

\subsection{Notation}
\label{sec:notation}

We introduce some notation that we will use throughout the paper. Write $\N := \{ 1,2,\ldots\}$, $\ZP := \{ 0,1,2,\ldots\}$, and $\RP := [0,\infty)$.
 The vertices of our graphs will be
associated with {\em sites} in a subset $S$ of an ambient $d$-dimensional space $(d \in \N)$.
Throughout we assume that $S \subset \R^d$ is compact, convex, and of positive $d$-dimensional Lebesgue measure (since $S$ is compact, it is a Borel set).
The location of the sites for the vertices will be
distributed according to
 a density function $f$  supported on $S$.
 Let $X_0, X_1, \ldots$ be independent random variables with density $f$, and for $n \in \N$
 set $\cX_n := \{ X_0, \ldots, X_n \}$.
For most of our main results, we will assume that $f$ is bounded
away from $0$ and $\infty$ on its support $S$:
\begin{equation}
\label{fcon}
0 < \inf_{x \in S} f(x) \leq \sup_{x \in S} f(x) < \infty.
\end{equation}

 We
 write $\| \, \cdot \, \|$ for the Euclidean norm on $\R^d$, and 
 $\rho(x , y) = \| x -y\|$ for the Euclidean distance
between $x$ and $y$ in $\R^d$. 
Denote by $B( x; r)$ the open
 Euclidean $d$-ball centred at $x \in \R^d$ with radius
$r >0$.
Throughout we understand $\log x$ to stand for $\max\{0, \log x\}$.
 Let $\# A$ denote the number of elements of a finite set $A$.

\subsection{On-line nearest-neighbour graph}
\label{sec:theong}

The  on-line nearest-neighbour
graph (ONG)
is constructed on
points arriving sequentially in $\R^d$ by
connecting each point
after the first
to its nearest (in the Euclidean sense)
predecessor.
The ONG is a natural and basic model of evolving  spatial
networks. Many real-world networks have spatial content and evolve over time by the addition
of new nodes and edges.  Often distances between nodes,
as measured in the ambient
space in which the
network is embedded,
are significant, and it is often desirable that edge-lengths be minimized:
this may be the case in electrical, communications, and
transport networks for example.
 The ONG is perhaps the simplest model
of a growing spatial network that captures some of the fundamental
properties that seem natural for such networks, while displaying
interesting mathematical behaviour and presenting challenges
for analysis.

The ONG is a special case (or limiting case) of several
models that have appeared in the literature,
including a version
of the `FKP' network model \cite{fkp,bbbcr}
and geometric preferential attachment models such as \cite{ms,ffv,jordan2010}
(specifically, it is the `$\alpha =  -\infty$' case of the model of Manna and Sen \cite{ms});
one contribution of the present paper is to explore this latter connection.
The ONG can also be viewed in the framework of the `minimal
directed spanning tree' \cite{survey}. The name `on-line nearest-neighbour graph'
was apparently introduced by Penrose in \cite{mdp}.

 In the ONG on $(X_0, \ldots, X_n)$, edges are added one by one, the $n$th edge ($n\in\N$)
between $X_n$ and its nearest neighbour among  $\cX_{n-1}$; with probability 1,
this nearest neighbour is unique, since  ties occur with probability $0$. In other words, writing
\begin{equation}
\label{nnnot}
 \eta_1 (n) := \argmin_{i \in \{ 0, \ldots, n-1 \}} \rho ( X_n , X_i ) \end{equation}
for the index of the (a.s.\ unique) nearest predecessor of $X_n$, the ONG on $( X_0, \ldots, X_n )$
consists of the edges
$( i, \eta_1 (i) )$ for $1 \leq i \leq n$; it is natural to view these as directed edges
when constructing the graph, but we also view them as undirected edges when convenient (e.g.\ when computing degrees).
We call $X_{\eta_1(n)}$ the {\em on-line nearest neighbour} of $X_n$.

Let $\deg_n (i)$ denote the degree of vertex $i$ in the ONG on $(X_0, \ldots, X_n)$, viewed as an undirected graph;
so this includes, for $i \neq 0$, the outgoing edge $(i, \eta_1 (i))$ in addition to any incoming edges $(j,i)$, $i < j \leq n$. 
Let $N^\mathrm{ONG}_n(k)$ denote the number of vertices with degree at least $k$ in the ONG on $(X_0, \ldots, X_n)$:
\[ N^\mathrm{ONG}_n(k) = \sum_{i=0}^n \1 \{ \deg_n (i) \geq k \} .\]
We study the {\em asymptotic degree sequence}, i.e., the asymptotic proportion of vertices with degree at least $k$ (for each $k$).
So we are interested in the asymptotic behaviour of $(n+1)^{-1}  N^\mathrm{ONG}_n(k)$. For simplicity, however, we state
our results for $n^{-1}  N^\mathrm{ONG}_n(k)$; the asymptotics of the two are clearly equivalent.

 Part of the statement of our main result on the ONG, Theorem \ref{ongdeg} below,
  is that $\lim_{n \to \infty} n^{-1} \Exp [ N^\mathrm{ONG}_n (k) ]$ {\em exists}
 for each $k$; this was stated, apparently without proof, in \cite[\S 2]{bbbcr}, but can be justified for the ONG
using stabilization arguments of Penrose \cite{mdp}, as we explain in Section \ref{sec:ongproof} below. Stabilization also gives an explicit description of the limit
 in terms of a version of the ONG defined on an infinite Poisson point process, as we describe next; in particular, the limit depends only on $d$ and not on $S$ or $f$.

 Let $\cH$ denote a unit-rate homogeneous Poisson point process on $\R^d \times [0,1]$;
  the $[0,1]$-valued coordinate
 can be thought of as a uniform random {\em mark}, associated to each Possion point in $\R^d$,
that will play the role of time in the finite construction of the ONG.
 For $u \in [0,1]$, let $\cH_u := \cH \cap ( \R^d \times [0,u])$, those Poisson points with marks in $[0,u]$.
For $x, y \in \R^d$ let $B_x (y)$ denote the
open Euclidean ball with centre $y$ whose boundary includes $x$.
 Given  $x \in \R^d$  and $u \in [0,1]$,
 let
\[ \xi  ( x, u ; \cH ) := 1 + \sum_{(y, v) \in \cH , \, v > u} \1 \{  \cH_v \cap ( B_x ( y ) \times [0,1] )   = \{ (y, v) \} \} .\]
It is a consequence of stabilization for the ONG (see \cite{mdp}) that $\xi (x, u ; \cH) < \infty$ a.s.\ for any $x \in \R^d$ and any $u \in (0,1)$.
We call $\xi  ( x, u ; \cH )$ the {\em degree} of $(x,u)$ in the {\em infinite Poisson on-line nearest-neighbour graph},
which is defined locally by joining each   point to the nearest Poisson point with mark equal to or less than the mark of the given point;
note that $(x,u)$ itself
need not be in $\cH$. Let $U$ denote a uniform $[0,1]$ random variable, independent of $\cH$.

\begin{theorem}
\label{ongdeg}
Let $d \in \N$. Suppose that \eqref{fcon} holds. Then for any $k \in \N$,
\begin{equation}
\label{deglim}
\lim_{n \to \infty} n^{-1} N^\mathrm{ONG}_n (k) = \lim_{n \to \infty} n^{-1} \Exp [ N^\mathrm{ONG}_n (k) ] = \Pr [ \xi ( 0, U ; \cH )  \geq k] =: \rho_k ,
\end{equation}
the first limit equality holding a.s.~and in $L^1$.
 Here   $\rho_k \in [0,1]$ are nonincreasing with $\rho_1 =1$, $\lim_{k \to \infty} \rho_k =0$, and $\sum_{k \in \N} \rho_k = 2$.
Moreover,
 there exist finite positive constants $A,A',C,C'$ such that, for all $k \in \N$,
 \begin{equation}
 \label{degbounds}
  A' \re^{-C'k} \leq \rho_k  \leq A \re^{-Ck} ,\end{equation}
and
\begin{equation}
\label{eq:rho_lower_bound}
\frac{1}{2} \log \left( 1 + (2^{2d} -1)^{-1} \right) \leq
\liminf_{k \to \infty} \left( - k^{-1} \log \rho_k \right) \leq
\limsup_{k \to \infty} \left( - k^{-1} \log \rho_k \right) \leq 1 .
\end{equation}
  Finally, there exists a constant $D < \infty$ for which, a.s., for all $n$ sufficiently large,
  \begin{equation}
  \label{ongmax}
  \max_{0 \leq i \leq n} \deg_n  (i) \leq D \log n .\end{equation}
\end{theorem}

This result extends a result of Berger {\em et al.} \cite{bbbcr}. Specifically, \cite[Theorem 3]{bbbcr} showed
\[
 A' \re^{-C'k} \leq \liminf_{n \to \infty} n^{-1} \Exp [  N^\mathrm{ONG}_n (k) ]  \leq \limsup_{n \to \infty} n^{-1}  \Exp [  N^\mathrm{ONG}_n (k) ]
 \leq A \re^{-Ck} ,   \]
in the special case where $d=2$ and $f$ is the indicator of the unit square $S = (0,1)^2$.
 Our proof
of   Theorem \ref{ongdeg}, which we give in Section \ref{sec:ongproof} below,
 is based in part on the proof of the analogous result in \cite{bbbcr}, with additional arguments required
 to obtain the existence of the limit  and the almost-sure convergence in   \eqref{deglim}.
Some extra work is also needed to obtain the quantitative bounds in 
  \eqref{eq:rho_lower_bound}: 
the $d=2$ case of the lower bound, $\frac{1}{2} \log \frac{16}{15}$,
is contained in the argument of \cite{bbbcr}; the other bounds are new.

\begin{remark}
In view of \eqref{eq:rho_lower_bound}, it is natural to conjecture that, for each $d \in \N$, 
\[   \lim_{k \to \infty} \left( - k^{-1} \log \rho_k \right) = \mu(d) \in (0, 1] \]
exists;
the upper bound of $1$ comes from \eqref{eq:rho_lower_bound}. In \cite[Section 7.6.5]{survey} it was conjectured that one might have $\mu(d) = \mu = 1$. 
The analogous but simpler, non-spatial, \emph{uniform attachment} model in which
 vertex $n$ is connected uniformly
at random to a vertex from $\{0,1,\ldots, n-1\}$ leads to an analogous
result with $\mu = \log 2$, as follows from the discussion in \cite[\S 4]{brst}.
The present authors suspect that $\mu(d)$ exists,
but think it unlikely that $\mu (d) \in \{1, \log 2\}$ for any $d \in \N$;
we conjecture, however, that $\lim_{d \to \infty} \mu (d) = \log 2$,
so we do not expect the lower bound in \eqref{eq:rho_lower_bound}, which tends to $0$ as $d \to \infty$, to be sharp. 
Simulations suggest that $\mu(1) \approx 0.79$, $\mu(2) \approx 0.77$, and $\mu(100) \approx 0.69$ (see Table \ref{table1}
for simulation results).
It may be possible to estimate $\mu (d)$
using
the infinite Poisson description of $\rho_k$.
\end{remark}

\begin{table}
\center
\footnotesize{
\begin{tabular}{c|cccccccccc}
degree & 1 & 2 & 3 &4 & 5 & 6 & 7 & 8 & 9 & 10  \\ 
\hline
$d=1$ &
$0.4728$ &
$0.2675$ &
$0.1394$ &
$0.0670$ &
$0.0304$ &
$0.0132$ &
$0.0056$ &
$0.0024$ &
$0.0001$ &
$0.0000$ \\
\hline
$d=2$ &
$0.4777$ & 
$0.2636$ & 
$0.1369$ & 
$0.0668$ &
$0.0308$ & 
$0.0137$ &
$0.0060$ &
$0.0026$ &
$0.0001$ &
$0.0000$ \\
\hline
$d=100$  & 
$0.4999$ & 
$0.2501$ & 
$0.1250$ & 
$0.0625$ &
$0.0312$ & 
$0.0156$ &
$0.0078$ &
$0.0039$ &
$0.0002$ &
$0.0001$ 
\end{tabular}
}
\caption{Estimated $\Pr[ \xi (0 , U ; \cH ) = k]$ for $1 \leq k \leq 10$, for $d \in \{1,2,100\}$. For each $d$,
the estimates are based on 
$500$ simulations with $n = 10^5$ for $f$ the uniform density on the $d$-dimensional torus. 
Values are given to 4dp; for $k \geq 11$ all values are
$0.0000$ to 4dp.}
\label{table1}
\end{table}

\subsection{Geometric preferential attachment graph}
\label{sec:gpagraph}

The version of the geometric preferential attachment (GPA)
model that we   study is as follows; often our  notation coincides with \cite{jordan2010}. We  define
a (random) sequence of finite graphs $G_n = (V_n, E_n)$, $n \in \N$.
The vertex set of $G_n$
is $V_{n} =  \{ 0, 1, \ldots, n\}$.
For $v \in V_n$, we denote by $\deg_n(v)$ the degree of $v$ in the GPA graph $G_n$ (viewed as an undirected graph);
this notation is the same as for degrees in the ONG, but the graph under consideration will be clear in context.
 
The construction uses an {\em attractiveness function} $F : (0,\infty) \to (0,\infty)$.
Recall that $X_0, X_1, \ldots$ are random sites in $S$.
There is some flexibility in
exactly how to start the construction, and one may start with some initial
fixed graph without changing any of our results. For notational simplicity, we start with an initial graph $G_1 = (V_1, E_1)$
consisting of vertices with labels $0$ and $1$ joined by a single edge, so $V_1 = \{0, 1\}$ and $E_1 = \{ (1,0) \}$.
(As in the ONG, there will be a natural direction associated to each edge by the construction, but we typically ignore these directions
when talking about properties of the graphs.)
Vertices $0$ and $1$ are associated with sites $X_0$ and $X_1$ in $S$, respectively.
 
We proceed via iterated addition of vertices to construct $G_{n+1}=(V_{n+1}, E_{n+1})$ from $G_n = (V_n,E_n)$,
$n \in \ZP$.
Given $G_{n}$, $n \in \N$,
  and the spatial locations $\cX_{n}$ of its vertices,
we add a vertex with label $n+1$ at site $X_{n+1} \in S$, and
we add a new edge $( n+1, v_{n+1} )$ where $v_{n+1}$ is chosen randomly from $V_{n}$
with distribution specified by
\begin{equation}
\label{rule}
  \Pr [ v_{n+1} = v \mid G_{n}, \cX_{n+1} ] = \frac{ \deg_{n} (v) F ( \rho ( X_{v}, X_{n+1} ) ) }{ D_{n} (X_{n+1} ) } , ~~~ v \in V_n,\end{equation}
where for $n \in \N$ and $x \in S$,
\[ D_n (x) := \sum_{v \in V_n } \deg_n (v)  F ( \rho (   X_v, x ) ) .\]

We call $G_n$ so constructed a GPA graph with attractiveness function $F$.
In \cite{jordan2010}, it was assumed that
$\int_S F( \rho(x,y) ) \ud y < \infty$, so that  the attractiveness function $F$
should not blow up too rapidly at $0$.
In this paper, our primary interest is in functions $F$
for which this condition is not satisfied.

\subsection{Strong geometric regime}
\label{sec:strong}

For $\gamma >1$, define $F_\gamma$ for $r>0$ by
\[ F_\gamma (r) := \exp \{ (\log (1/r) )^\gamma \}. \]
Note that $F_\gamma(r)$ blows up at $0$ faster than $r^{-s}$ for any power $s$.
Recall  that the convention $\log x \equiv \max\{ 0 ,\log x\}$ is in force, so
  $F_\gamma (r) = 1$ for $r \geq 1$. Also, $F_\gamma (r)$ is strictly decreasing
  for $r \in (0,1)$, with $F_\gamma (r) \to \infty$ as $r \downarrow 0$.

Our main result in this setting   (i) gives an almost-sure degree bound
analogous to  \eqref{ongmax} above for the ONG, and (ii) shows
that 
the limiting degree sequence
for the GPA graph is the same as for the ONG,  for a strong enough geometric component to the interaction (under the
condition $\gamma > 3/2$).
Let $N^\mathrm{GPA}_n (k)$ denote the number of vertices with degree at least $k$ in the GPA graph $G_n$.

\begin{theorem}
\label{thm1}
Suppose that \eqref{fcon} holds and that
$F  = F_\gamma $ for some $\gamma > 1$.
\begin{itemize}
\item[(i)] For any $\nu \in (0,1)$
with $\nu > 2-\gamma$,   a.s., for all   $n$ sufficiently large,
\begin{equation}
\label{gpadeg}
 \max_{0 \leq i \leq n} \deg_n (i) \leq   \exp \{   (\log n)^\nu \} .\end{equation}
\item[(ii)] Suppose that $\gamma > 3/2$. Then $\lim_{n \to \infty} \Pr [ v_n = \eta_1 (n) ] = 1$ and
 the expected number of vertices in the GPA graph that are joined to a vertex other than
their on-line nearest neighbour satisfies
\begin{equation}
\label{nnlim}
\lim_{n \to \infty} 
n^{-1}
\Exp \sum_{i=1}^n \1 \{ v_i \neq \eta_1 (i) \} = 0 .
\end{equation}
Moreover, for any $k \in \N$,
\begin{equation}
\label{gpa_to_ong}
 \lim_{n \to \infty} n^{-1}  N^\mathrm{GPA}_n (k)  =  \lim_{n \to \infty} n^{-1} \Exp [ N^\mathrm{GPA}_n (k) ] = \rho_k ,\end{equation}
the first limit equality holding in $L^1$,
 where $\rho_k$ is the limiting degree sequence
for the ONG as given in Theorem \ref{ongdeg}.
\end{itemize}
\end{theorem}
 
We give the proof of Theorem \ref{thm1} in Section \ref{sec:proofs1}.
 
\begin{remark}
The statements \eqref{nnlim} and \eqref{gpa_to_ong} are $L^1$ convergence results, and hence imply
convergence in probability for the quantities concerned. It would be of interest
 to extend \eqref{nnlim} and \eqref{gpa_to_ong} to almost sure convergence. One possible approach would be via a concentration argument similar to that we use in the case of the ONG (see
Lemma \ref{ongconc} below), but this seems to require better tail bounds on large degrees in the GPA graph.
\end{remark}

\begin{conjecture}
We suspect that the conclusion of Theorem \ref{thm1}(ii) is valid for {\em any} $\gamma > 1$.
\end{conjecture}

\subsection{Intermediate regime: power-law attractiveness}
\label{sec:power-law}

Take $F(r) = r^{-s}$ for $s \in (0,\infty)$.
The next result contrasts with \eqref{nnlim} in the strong geometric attraction regime,
and shows that in this case, in expectation,
there is a non-negligible proportion of vertices not
connecting to their nearest neighbour.

\begin{theorem}
\label{thm:int2}
Suppose that \eqref{fcon} holds and   $F(r) = r^{-s}$ for $s \in (0,\infty)$.
Then $\limsup_{n \to \infty} \Pr [ v_n = \eta_1 (n) ] < 1$ and
 the expected number of vertices in the GPA graph that are joined to a vertex other than
their on-line nearest neighbour satisfies
\begin{equation}
\label{notnn}
\liminf_{n \to \infty} 
n^{-1}
\Exp \sum_{i=1}^n \1 \{ v_i \neq \eta_1 (i) \} > 0 .
\end{equation}
\end{theorem}

Next we examine the degree sequence of the graph. It was proved in Theorem 2.1 of \cite{jordan2010} 
that in the case $s \in (0,d)$, under certain conditions on $S$ and $f$,
the degree distribution of the GPA graph converges to
a power-law distribution, as in the Barab\'asi--Albert model:
  $\lim_{n\to\infty} n^{-1} \Exp [ N^{\rm GPA}_n (k) ] =r_k$
where $r_k \sim 2 k^{-2}$ as $k \to \infty$. 

The next result shows contrasting behaviour when $s > d$: we give a stretched exponential
upper bound for
the tail of the
degree distribution, which thus decays faster than any power law.

\begin{theorem}
\label{thm:int1}
Suppose that \eqref{fcon} holds and  $F(r) = r^{-s}$ for $s > d$.
For any $\gamma \in (0, \frac{s-d}{2s-d} )$, there exists a constant $C < \infty$ such that,
for all $k$,
\[ \limsup_{n \to \infty} n^{-1} N_n^{\rm GPA} (k) \leq C \re^{-k^\gamma} , \as,
 ~~\text{and}~~   \limsup_{n \to \infty} n^{-1} \Exp [ N_n^{\rm GPA} (k) ] \leq C \re^{-k^\gamma} .\]
\end{theorem}

This result confirms the presence of a phase transition in the character of the degree distribution
at $s =d$, as intimated in \cite[\S 5]{jordan2010}
and in line with the $d \in \{1,2\}$ 
simulation results of Manna and Sen \cite{ms} (who themselves
actually conjectured that the phase transition point was $s=d-1$). The stretched exponential
for $s>d$ is also consistent with the simulation-based observations of \cite{ms}. We remark that
as $s \to \infty$, Theorem \ref{thm:int1} gives
an upper bound of order almost $\re^{-\sqrt k}$;
it is not clear 
whether this is sharp, although Manna and Sen \cite[p.~3]{ms} do suggest that one might
 expect instead to approach a genuinely exponential tail in the limit $s \to \infty$.

\section{Preliminaries to the proofs}
\label{sec:prelim}

First we state a basic property of the set $S$, under our standing assumptions. Let $\omega_d$ be the volume of the unit-radius Euclidean $d$-ball,
and set   $\mathrm{diam}(S) := \sup_{x, y \in S} \rho(x,y)$.

\begin{lemma}
\label{balls}
There exists $\delta_S >0$ such that, for all $r \in [0, \mathrm{diam} (S) ]$,
\[ \inf_{x \in S} | B (x ; r) \cap S | \geq \delta_S \omega_d r^d .\]
\end{lemma}
\begin{proof}
Since $S$ is convex, compact, and of positive measure, there exist $x_0 \in S$ and $r_0 >0$ such that $B (x_0; r_0)$ is contained in the interior of $S$.
It suffices to suppose that either (i) $\rho ( x, x_0) \geq 2 r_0$, or (ii) $\rho (x, x_0 ) \leq r_0 / 2$.
To see this, suppose that $r_0 /2 < \rho ( x, x_0) < 2 r_0$. Then  we may carry out the argument for case (i) after having replaced $r_0$ by $r_0/4$,
introducing only a constant  multiplicative factor into the argument.

So now suppose that (i) holds.
For $r \leq r_0$,
let $C (x , r)$
denote the cone with apex $x$, axis passing through $x_0$, and half-angle $\theta (x, r) = \sin^{-1} ( r  / \rho ( x , x_0) )$.
 Since $\rho ( x , x_0 ) \leq \mathrm{diam} (S)$,
 $\theta (x, r) \geq \theta (r) := \sin^{-1} ( r / \mathrm{diam} (S) )$.
By construction and convexity of $S$,
$C (x , r) \cap S$ contains the cone segment $\{ y \in C (x,r) : \rho ( x, y) \leq \rho ( x, x_0 ) \cos \theta (x,r) \}$. So, if  $\rho (x, x_0 ) \geq 2r_0$,
then $B ( x ; r) \cap S$ contains
the cone segment $\{ y \in C (x , r) : \rho ( x, y) \leq r \wedge  r_0   \}$, which has volume bounded below by  $c_d \theta(r)^{d-1} r$, provided $r \leq r_0$,
where $c_d >0$ is an absolute constant. Hence $| B ( x ; r) \cap S |$ is bounded below by a constant times $r^d$, for all $r \leq r_0$.
On the other hand,
if $r \in (r_0, \mathrm{diam} (S) )$ we may use the lower bound $c_d \theta (r_0)^{d-1} r_0 \geq c'_d (r_0 / \mathrm{diam} (S))^d r^d$ for $c'_d >0$ not depending on $r$.
So again $| B ( x ; r) \cap S |$ is bounded below by a constant times $r^d$.

Finally, in case (ii), we have that $B ( x ; r) \cap S$ contains the ball $B (x ; r \wedge (r_0 /2))$, and a similar argument to that for part (i) completes the proof.
\end{proof}

We next give some basic results
on nearest-neighbour distances.
For $n \in \N$,
let
\[ Z_n := \rho ( X_n ; \cX_{n-1} ) := \min_{0 \leq i \leq n-1} \rho (X_n , X_i) = \rho ( X_n, X_{\eta_1(n)} ),  \]
  the distance
from $X_n$ to its on-line nearest neighbour. Write $x^+ := x \1 \{ x > 0 \}$.

\begin{lemma}
\label{nnlem}
Let $\delta_S >0$ be the constant in Lemma \ref{balls}.
\begin{itemize}
\item[(i)] Suppose that $\inf_{x \in S} f(x) = \lambda_0 >0$. Then for $r>0$,
\begin{equation}
\label{nntail}
 \Pr [ Z_n \geq r ] \leq ( 1 - \delta_S \lambda_0 \omega_d r^d )^n \1 \{ r \leq \mathrm{diam} (S) \} .\end{equation}
 \item[(ii)]
  Suppose that $\sup_{x \in S} f(x) = \lambda_1 < \infty$.  Then for $r>0$,
\begin{equation}
\label{nntail2}
 \Pr [ Z_n \geq r ] \geq ( ( 1 - \lambda_1 \omega_d r^d )^+ )^n \1 \{ r \leq \mathrm{diam} (S) \} .\end{equation}
 \end{itemize}
\end{lemma}
\begin{proof}
Conditional on $X_n$, we have, for any $r >0$, a.s.,
\begin{align}
\label{nn0}
\Pr [ Z_n \geq r \mid X_n ] & = \Pr [ S \cap B( X_n ; r) \cap \cX_{n-1}   = \emptyset \mid X_n ] \nonumber\\
& = \left( 1 - \int_{S \cap B( X_n ; r)} f (x) \ud x \right)^n . \end{align}
Note that   $\Pr [ Z_n > \mathrm{diam} (S) ] = 0$, so it suffices to suppose that $r \leq \mathrm{diam} (S)$.
Using Lemma \ref{balls} we have that, for $\delta_S > 0$,
\begin{equation}
\label{vol0}
 \delta_S \omega_d r^d  \leq
  | S \cap B (X_n ; r ) | \leq \omega_d r^d , \as , \end{equation}
  for all $r \leq \mathrm{diam} (S)$.
  It   follows from \eqref{nn0}  that, if $\inf_{x \in S} f(x) = \lambda_0 >0$,
\[
\Pr [ Z_n \geq r \mid X_n ] \leq \left( 1 - \lambda_0 | S \cap B( X_n ; r) |  \right)^n  ,\]
which, with  the first inequality in \eqref{vol0}, gives part (i).
Under the condition $\sup_{x \in S} f(x) = \lambda_1 < \infty$, we obtain part (ii) similarly from \eqref{nn0} with the second inequality
in \eqref{vol0}.
\end{proof}

Next we state a simple but useful result on degrees in our graphs.

\begin{lemma}
\label{deglem}
In either the GPA graph or the ONG, writing $N_n(k)$ for $N_n^\mathrm{GPA} (k)$ or $N_n^\mathrm{ONG} (k)$
as appropriate, we have that for any $k  \in \N$ and any $n$,
\[ N_n (k) \leq 2n /k , \as \]
\end{lemma}
\begin{proof}
This is basically Markov's inequality.
The property of the graphs that we use is simply that on $n+1$ vertices there are $n$ edges present, and all vertices have degree at least $1$.
By the degree sum formula,
\[ 2 n = \sum_{i=0}^n \deg_n (i) = \sum_{i=0}^n \sum_{k \geq 1} \1 \{ \deg_n (i) \geq k \} = \sum_{k \geq 1} N_n (k) ,\]
interchanging the order of summation. So for any $k_0 \in \N$,
\[ 2n \geq \sum_{k =1} ^{k_0} N_n (k) \geq k_0 N_n (k_0) ,\]
since $N_n (k)$ is nonincreasing in $k$.
\end{proof}

\section{Proofs for strong geometric regime}
\label{sec:proofs1}

In this section we give the proofs of our results from Section \ref{sec:strong}.
We start by outlining the idea behind the proof of Theorem \ref{thm1}.
The core of the argument is to show that $X_n$ is joined to its on-line nearest neighbour with probability $1 - o(1)$ (cf Lemma \ref{lem4} below).
By \eqref{rule}, the probability that $X_n$ is joined to its on-line nearest neighbour $X_{\eta_1(n)}$ satisfies
\[ \Pr [ v_n = \eta_1 (n) \mid G_{n-1} , \cX_n ] = \frac{ \deg_{n-1} (\eta_1 (n))  F ( Z_n )}{D_{n-1} (X_n) }   .\]
For $F = F_\gamma$,   the fact that  $F_\gamma$ is decreasing and the crude bound $\deg_{n-1} (i) \leq n$ give
\[ D_{n-1} (X_n ) = \sum_{i=0}^{n-1} \deg_{n-1} (i) F_\gamma ( \rho( X_i, X_n ) ) 
\leq n^2 F_\gamma ( W_n) + \deg_{n-1} (\eta_1 (n))  F_\gamma ( Z_n ) ,\]
where   $W_n$ is the distance from $X_n$ to its {\em second} nearest neighbour among $\cX_{n-1}$, so
\[ \Pr [ v_n = \eta_1 (n) \mid G_{n-1} , \cX_n ] \geq 1 - \frac{n^2 F_\gamma(W_n)}{F_\gamma(Z_n)} .\]
With probability $1-o(1)$, $W_n > Z_n + \theta_n$ where $\theta_n = o(n^{-1/d})$, so to show $\Pr [ v_n = \eta_1 (n)  ] = 1 - o(1)$
it suffices to show that,
\[ \frac{n^2 F_\gamma (Z_n + \theta_n )}{F_\gamma (Z_n)} \to 0,\]
in probability, as $n \to \infty$. A computation using Taylor's formula shows that this holds provided $\gamma >2$. To improve on this argument
we need (i) to   control  the degrees of the vertices, and (ii) to control the number of `plausible alternatives' for $v_n$.

For $\nu \in (0,1)$ and $n \geq 2$ set
$\beta (n, \nu) := n^{-1/d} \exp \{ ( \log n )^\nu \}$, and let
\[ E (n, \nu) := \{ \rho ( X_{v_n} , X_n ) \geq \beta (n , \nu) \} ,\]
 the event that the edge from vertex $n$
connects to any vertex  
outside $B ( X_n ; \beta (n ,\nu) )$.

\begin{lemma}
\label{lem1}
Suppose that $F  = F_\gamma$ for some $\gamma >1$ and that $\nu \in (0,1)$ with $\nu > 2-\gamma$.
Suppose that $\inf_{x \in S} f(x)   >0$.
Then for any $p < \infty$, as $n \to \infty$,
\[  \Pr [ E (n, \nu ) ] =  O ( \exp \{ - \gamma d^{1-\gamma} (1 + o(1) ) (\log n)^{\gamma + \nu -1} \} ) = O( n^{-p} ). \]
\end{lemma}
\begin{proof}
Note that
for any $\nu \in (0,1)$,
\begin{align}
\label{eq2}
 F_\gamma (\beta (n , \nu)  ) & = \exp \left\{ \left( d^{-1} \log n - (\log n)^\nu  \right)^\gamma \right \} \nonumber\\
 & = \exp \left\{ d^{-\gamma} (\log n)^\gamma - \gamma d^{1-\gamma} (1+o(1)) (\log n)^{\gamma + \nu -1}
 \right\} .
 \end{align}

Given $\inf_{x \in S} f(x) = \lambda_0  >0$,  we obtain from \eqref{nntail} that
\begin{equation}
\label{eq3}
 \Pr [ Z_n > \beta(n,\nu)  ] = O ( \exp \{ - \delta_S \lambda_0 \omega_d \exp \{ d (\log n)^\nu \}   \} )
= O (  \exp \{ - (\log n)^K \} ) ,\end{equation}
for any $K < \infty$,
since $\exp \{ (\log n)^\nu \}$ grows faster than any power of $\log n$.

Fix $\nu \in (0,1)$ and choose $\nu' \in (0,\nu)$.
Then
\begin{equation}
\label{eq4}
 \Pr [ E (n,\nu) ] \leq \Pr [ Z_n > \beta (n ,\nu')  ]
+ \Pr [ E (n,\nu) \mid Z_n \leq \beta (n, \nu')  ] .\end{equation}
Suppose that $Z_n \leq \beta (n, \nu')$.
Then, if the  nearest neighbour of $X_n$
among $\cX_{n-1}$ is $X_{\eta_1(n)}$,
since $F_\gamma (r)$ is nonincreasing in $r >0$,
\[ \deg_{n-1} ( \eta_1 (n) ) F_\gamma ( \rho ( X_{\eta_1(n)}, X_n ) ) \geq F_\gamma (Z_n) \geq
F_\gamma ( \beta (n, \nu') ) ,\]
so that $D_{n-1}(X_n) \geq F_\gamma (\beta (n, \nu'))$, given $Z_n \leq \beta (n, \nu')$.
On the other hand, any vertex $j < n$ with $X_j \notin B( X_n ; \beta (n, \nu) )$
has
\[ \deg_{n-1} (j) F_\gamma ( \rho (X_j, X_n ) ) \leq n F_\gamma (\beta (n,\nu) ) ,\]
using the crude bound $\deg_{n-1} (j) \leq n$.
Hence, by \eqref{rule} and \eqref{eq2},
\begin{align}
\label{eq5}
&  \Pr [ E (n,\nu) \mid Z_n \leq \beta (n, \nu') ]  = \sum_{j = 0}^{n-1} \Pr [ \{ v_n = j \}  \cap E (n , \nu) \mid Z_n \leq \beta (n, \nu') ] \nonumber\\
 & {}~~  \leq
\frac{n^2 F_\gamma (\beta ( n, \nu) )}{F_\gamma ( \beta(n , \nu') )} \nonumber\\
 & {}~~  = O \left(  \exp \left\{ 2 \log n - \gamma d^{1-\gamma} (1+o(1)) \left( (\log n)^{\gamma + \nu -1}
 - (\log n)^{\gamma + \nu' -1} \right) \right\} \right) \nonumber\\
 & {}~~  = O \left(  \exp \left\{ - \gamma d^{1-\gamma} (1+o(1))  (\log n)^{\gamma + \nu -1}
   \right\} \right)
  ,\end{align}
provided that $\gamma + \nu -1 > 1$, i.e., $\nu > 2 -\gamma$,
which we can  ensure by choosing $\nu \in (0,1)$ close enough to $1$ since $\gamma >1$.
The result now follows from \eqref{eq3}, \eqref{eq4} and \eqref{eq5}.
\end{proof}

The next result   is a bound on degrees that amounts to  Theorem \ref{thm1}(i), and which
will also be an ingredient in our proof of Theorem \ref{thm1}(ii).

\begin{lemma}
\label{lem2}
Suppose that \eqref{fcon} holds and   $F = F_\gamma$ for some $\gamma >1$.
Then for any $\nu \in (0,1)$
with $\nu > 2-\gamma$,   a.s., for all but finitely many $n \in\N$, \eqref{gpadeg} holds.
\end{lemma}
\begin{proof}
Let  $\lambda_0 = \inf_{x \in S} f(x)$ and $\lambda_1 = \sup_{x \in S} f(x)$; given \eqref{fcon}, $0 < \lambda_0 \leq \lambda_1 < \infty$.
Let $\gamma >1$ and $\nu > 2-\gamma$.
By Lemma \ref{lem1},
$\Pr [ E ( j, \nu) ] = O( j^{-2} )$.
Hence, by the Borel--Cantelli lemma, for only finitely many $j \in \N$ does the vertex $j$
connect to a vertex $i < j$ with $\rho(X_i, X_j ) \geq \beta (j,\nu)$.
It follows that there exists some finite random variable $D_\nu = 1 + \sum_{j=1}^\infty \1 ( E (j,\nu))$ such that, for all $n \in \N$
and all $i \in \{0,1,\ldots,n\}$,
\begin{align*}
\deg_n (i) \leq D_\nu + \sum_{j=i+1}^n \xi_{i,j} ,\end{align*}
where we set $\xi_{i,j} := \1 \{ \rho (X_j, X_i ) \leq \beta (j,\nu) \}$ for $i \neq j$ and $\xi_{i,i} := 0$.
Hence
\begin{align}
\label{maxdeg}
\max_{0 \leq i \leq n} \deg_n (i) \leq D_\nu + \max_{0 \leq i \leq n} \sum_{j=1}^n \xi_{i,j} .\end{align}
For   fixed $i$, conditional on $X_i$,
 the $n-1$ terms $\xi_{i,j}$ with $j \neq i$
  in the sum on the right-hand side of \eqref{maxdeg}
 are independent and $\{0,1\}$-valued,
and an appropriate version
of Talagrand's inequality (see e.g.\ \cite[p.\ 81]{mr}) will show that their sum is  concentrated around its mean
(in fact, we only need an upper bound here).
Specifically, we have for $n \in \N$,
\begin{align}
\label{tal1} \Exp  \sum_{j=1}^n \xi_{i,j} & = \sum_{j=1}^n
\Pr [ X_i \in B ( X_j ; \beta (j,\nu) ) ]   = \Theta \left( \sum_{j=1}^n \beta (j,\nu)^d \right)   ,\end{align}
uniformly for $i \in \{1,\ldots,n\}$, where the implicit constants depend on $S$, $\lambda_0$ and $\lambda_1$ (we use Lemma \ref{balls} here).
We claim that
\begin{equation}
\label{sumbeta}
\sum_{j=1}^n \beta (j,\nu)^d = \exp \{ d (\log n)^\nu ( 1+ o(1) ) \} . \end{equation}
To verify \eqref{sumbeta}, we combine the upper bound given by
\[ \sum_{j=1}^n \beta (j,\nu)^d \leq \sum_{j=1}^n \frac{1}{j} \exp \{ d (\log n)^\nu \}   \leq (1 + \log n) \exp \{ d (\log n)^\nu \} ,\]
with the lower bound given by
\[ \sum_{j=1}^n \beta (j,\nu)^d \geq \sum_{j = \lceil n/2 \rceil}^n \frac{1}{n} \exp \{ d (\log (n/2) )^\nu \} \geq \frac{1}{2} \exp \{ d (\log (n/2) )^\nu \}  ,\]
since the last sum contains $n +1 - \lceil n/2 \rceil \geq n/2$ terms.
From \eqref{tal1} and \eqref{sumbeta}, we have $\Exp \sum_{j=1}^n \xi_{i,j} = \exp \{ d (\log n)^\nu ( 1+ o(1) ) \}$.
  Talagrand's inequality implies that for   all $n$,
\begin{align*}
\max_{0 \leq i \leq n}
\Pr \left[  \sum_{j=1}^n \xi_{i,j} >   \exp \{ 2 d (\log n)^\nu \} \right] &
\leq O ( \exp \{ - \re^{  d    (\log n)^\nu  } \} ) ,\end{align*}
which is $O (n^{-3} )$,
say, so that Boole's inequality yields
\[ \Pr \left[  \max_{0 \leq i \leq n} \sum_{j=1}^n \xi_{i,j} >   \exp \{ 2 d (\log n)^\nu \} \right] = O (n^{-2} ) .\]
Now another application of the Borel--Cantelli lemma together with \eqref{maxdeg} completes the proof of the lemma,
noting that $\nu > 2-\gamma$ was arbitrary.
\end{proof}

The main step remaining in the proof of Theorem \ref{thm1} is the following.

\begin{lemma}
\label{lem4}
Suppose that \eqref{fcon} holds and  $F  = F_\gamma $ for some $\gamma > 3/2$.
Then $\Pr [ v_n \neq \eta_1 (n) ] \to 0$  as $n \to \infty$.
\end{lemma}

Before giving the proof of Lemma \ref{lem4}, we introduce some notation for dealing with conditional probabilities
that we will also use later on.
Let $\cF_n = \sigma ( \cX_n , v_2, v_3, \ldots , v_{n-1} )$,
the $\sigma$-algebra generated by the spatial locations of the vertices
up to and including $X_n$ and by the edge choices made on previous steps. Then $\deg_{n-1} (i) = \1 \{ i \neq 0 \} + \sum_{j = i+1}^{n-1} \1 \{ v_j = i\}$,
$D_{n-1} (x)$,   and  $X_0, \ldots, X_n$ are all $\cF_n$-measurable, and \eqref{rule} can be expressed as
\begin{equation}
\label{rule2}
 \Pr [ v_n = v \mid \cF_n ] = \frac{ \deg_{n-1} (v) F (\rho ( X_v, X_n ))}{D_{n-1} (X_n)} , ~~~ v \in \{0,\ldots, n-1\} .\end{equation}

\begin{proof}[Proof of Lemma \ref{lem4}.]
Again, by \eqref{fcon},
 $\lambda_0 = \inf_{x \in S} f(x) >0$ and $\lambda_1 = \sup_{x \in S} f(x) < \infty$.
Take a sequence of positive numbers $\theta_n$ with $\theta_n = o(n^{-1/d})$,
and, given $X_n$ and $Z_n$, define the shells $A_n := B (X_n ; Z_n + \theta_n) \setminus B(X_n ; Z_n)$.
Let $a_n := \# ( A_n \cap \cX_{n-1} \setminus \{ X_{\eta_1(n)} \} )$, the number
of predecessors to $X_n$, other than its on-line nearest neighbour,
 inside $A_n$.

Conditional on $X_n$ and $Z_n$, the points of $\cX_{n-1} \setminus \{ X_{\eta_1 (n)} \}$
are independent and identically distributed on $S \setminus B( X_n; Z_n)$ with conditional distribution
given for measurable $\Gamma \subseteq S \setminus B( X_n; Z_n)$ by
$\Pr [ \, \cdot  \in \Gamma ] = \int_{\Gamma} g_n ( x) \ud x$, where
\[ g_n (x) = \frac{f(x)}{\Pr [ X_0 \in S \setminus B( X_n; Z_n) \mid X_n , Z_n ]} .\]
Note that, a.s.,
\[ \Pr [ X_0 \in S \setminus B( X_n; Z_n) \mid X_n , Z_n ] = 1 - \int_{S \cap B (X_n ; Z_n )} f(x) \ud x \geq 1 - \lambda_1 \omega_d Z_n^d \geq \frac{1}{2} ,\]
provided $Z_n \leq ( 2 \lambda_1 \omega_d)^{-1/d}$.
Moreover,   $S \cap A_n  $ has volume bounded above by
\[ \omega_d (Z_n + \theta_n)^d - \omega_d Z_n^d \leq C_d \theta_n ( \theta_n^{d-1} + Z_n^{d-1} ) ,\]
for some finite constant $C_d$ depending only on $d$.
Hence, conditional on $X_n$ and $Z_n$, each of the $n-1$ points
  $X_0, \ldots, X_{n-1}$, excluding  $X_{\eta_1(n)}$, lands in $A_n$ with probability at most
  \[ \frac{ \int_{S \cap A_n} f(x) \ud x }{ \Pr [ X_0 \in S \setminus B( X_n; Z_n) \mid X_n , Z_n ] }   \leq   2 \lambda_1 C_d  \theta_n ( \theta_n^{d-1} + Z_n^{d-1} )
    + \1 \{ Z_n > ( 2 \lambda_1 \omega_d)^{-1/d} \} .\]
  It follows that
\[ \Exp [ a_n \mid Z_n ] \leq 2 \lambda_1 C_d  n \theta_n ( \theta_n^{d-1} + Z_n^{d-1} )  + n \1 \{ Z_n > ( 2 \lambda_1 \omega_d)^{-1/d} \} .\]
Taking expectations and using \eqref{nntail} we have
$n \Pr [ Z_n > ( 2 \lambda_1 \omega_d)^{-1/d} ] = o(1)$, while, for any $\alpha >0$, by another application of \eqref{nntail},
\[ \Exp [ Z_n ^\alpha ] = \int_0^\infty \Pr [ Z_n > r^{1/\alpha} ] \ud r
\leq \int_0^\infty \exp \{ - C n r^{d/\alpha} \} \ud r ,\]
for some $C<\infty$, which gives $\Exp [ Z_n ^\alpha ] = O ( n^{-\alpha /d} )$.
Hence 
\[ \Exp [ a_n ] = O (\theta_n^d n ) + O ( \theta_n n^{1/d} ) + o(1) = o(1),\]
provided $\theta_n = o(n^{-1/d})$,
so that, by Markov's inequality,
$\Pr [ a_n > 0 ]  \leq \Exp [ a_n ] = o(1)$.

Now we 
condition on the whole of $\cF_n$.  
Again take $\beta(n, \nu) = n^{-1/d} \exp \{ (\log n)^\nu \}$.
Let $E'_n$ denote the event that $X_n$ is joined to a point outside
$B(X_n; Z_n + \theta_n)$:
\[ E'_n := \{ \rho ( X_{v_n} , X_n ) \geq Z_n + \theta_n \}.\]
Also, for a constant $b >1$ (which later we will choose to be large), set
\[ E''_n := \{ Z_n \leq b^{-1} n^{-1/d}  \} \cup \{ Z_n \geq b n^{-1/d} \} .\]
Finally, define the event (for another constant $C$ to be chosen later)
\[ E_n''' := \left \{ \# \left( \cX_{n-1}  \cap B (X_n ; \beta(n, \nu) ) \right)
\geq C  \exp \{ d (\log n)^\nu \} \right\} .\]
The ball $B ( X_n ;    \beta(n, \nu))$ has volume bounded above by $\omega_d n^{-1} \exp \{ d (\log n)^\nu \}$.
The events $\{ X_j \in B ( X_n ;    \beta(n, \nu)) \}$, $0 \leq j \leq n-1$ are independent each with probability
at most $\lambda_1 \omega_d n^{-1} \exp \{ d (\log n)^\nu \}$, so
$\# \left( \cX_{n-1}  \cap B (X_n ; \beta(n, \nu) ) \right)$
is stochastically dominated by a binomial $(n, \lambda_1 \omega_d n^{-1} \exp \{ d (\log n)^\nu \})$ random variable.
Standard binomial tail bounds show that, for an appropriate $C < \infty$,
$\Pr [ E_n '''] = o(1)$.

On $\{a_n = 0 \} \cap ( E_n' )^{\rm c}$, $X_n$ is necessarily connected to its on-line
nearest neighbour, so that
 the probability that $X_n$ is connected to a point other than its
on-line nearest neighbour satisfies
\begin{align}
\label{eq60}
\Pr [ v_n \neq \eta_1 (n) \mid \cF_n ] & \leq \Pr [ E'_n \mid  \cF_n ] \1( \{ a_n = 0 \} \cap (E''_n)^{\rm c} \cap (E'''_n)^{\rm c} ) \nonumber\\
& \quad ~~{} + \1 \{ a_n > 0 \} + \1( E''_n ) + \1( E'''_n ).\end{align}
For any $\eps >0$, 
Lemma \ref{nnlem}  shows that
we can choose $b$ and $n_0$ sufficiently large so that $\Pr [ E_n'' ] < \eps$ for all
$n \geq n_0$. We have already seen that $\Pr [ a_n > 0 ] = o(1)$ and $\Pr [ E'''_n ] = o(1)$.
We also claim that
\begin{equation}
\label{eq61}
\Pr [ E'_n \mid  \cF_n ] \1( \{ a_n = 0 \} \cap (E''_n)^{\rm c} \cap (E'''_n)^{\rm c} ) = o(1), \as \end{equation}
The bounded convergence theorem implies that the expectation of this last quantity is also $o(1)$, so taking expectations in \eqref{eq60} we
see that for any $\eps >0$, we may choose $b$ such that, for all $n$ large enough, $\Pr [ v_n \neq \eta_1 (n) ] \leq \eps$.
This gives the statement in the lemma.

It remains to prove the claim \eqref{eq61}. First we note that
\[ D_{n-1} ( X_n ) \geq \deg_{n-1} ( \eta_1(n)  ) F_\gamma ( \rho (X_{\eta_1(n)}, X_n ) ) \geq F_\gamma (Z_n). \]
On the other hand, on $\{
a_n =0 \}$, any alternative $X_j$
to $X_{\eta_1(n)}$ among $\cX_{n-1}$
is at distance at least $Z_n + \theta_n$ from $X_n$,
so that for $j \neq \eta_1 (n)$,
\[ \deg_{n-1} ( j ) F_\gamma (\rho(X_j, X_n)) \leq   \exp \{ (\log n)^\nu \} F_\gamma ( Z_n + \theta_n ) , \as, \]
for all $n$ large enough,
by Lemma \ref{lem2}, provided $\nu \in (0,1)$ with $\nu > 2 -\gamma$.

On $(E_n''')^{\rm c} \cap \{ a_n = 0\}$, the contribution of points
inside $B (X_n ;   \beta (n,\nu) )$, other than   $X_{\eta_1(n)}$,
  to $D_{n-1} (X_n)$ is bounded above by
\[   C \exp \{ 2d (\log n)^\nu \} F_\gamma ( Z_n + \theta_n ) ,\]
since there are at most $O ( \exp \{ d (\log n)^\nu \} )$ of these points, their
degrees are at most $O ( \exp \{ (\log n)^\nu \} )$, a.s., by Lemma \ref{lem2},
and they are all at distance at least $Z_n +\theta_n$ from $X_n$.
Moreover, similarly to as in the proof of Lemma \ref{lem1},
the contribution to $D_{n-1} (X_n)$ from any points outside $B(X_n; \beta (n,\nu) )$
is at most $n^2 F_\gamma( \beta (n,\nu))$.

So from \eqref{rule2} we have, on $\{ a_n =0 \} \cap (E_n''')^{\rm c}$, for all $n$ large enough,
\[ \Pr [ E'_n \mid  \cF_n ]
\leq \frac{C \exp \{ 2d (\log n)^\nu \} F_\gamma ( Z_n + \theta_n )
 +n^2 F_\gamma(\beta (n,\nu))}{F_\gamma (Z_n) } .\]
 Here, similarly to \eqref{eq5},
 \[ \frac{n^2 F_\gamma (\beta (n,\nu))}{F_\gamma (Z_n)} = O ( \exp \{ - c (\log n)^{\gamma + \nu -1} \} ) ,\]
 for some $c>0$, as long as $\nu > 2 -\gamma$.
Also we have that, on $(E_n'')^{\rm c}$,
\begin{align*} \frac{F(Z_n + \theta_n  )}{F(Z_n)}
& =
\exp \left\{ (\log (1/Z_n))^\gamma \left( \left( 1 + \frac{ \log ( 1 + (\theta_n /Z_n) )}{\log Z_n } \right)^\gamma -1 \right) \right\}
\\
& = \exp \left\{ - c ( \log n)^{\gamma -1} n^{1/d} \theta_n (1+o(1)) \right\},\end{align*}
provided $\theta_n = o(n^{-1/d})$.
In particular, for $\gamma -1 > \nu$, we can choose $\theta_n = n^{1/d} (\log n)^{1-\gamma+\nu+\eps}$
for some $\eps>0$ and $1-\gamma+\nu +\eps <0$.
The constraints $\gamma -1 > \nu$ and $\nu > 2-\gamma$ entail $\gamma > 3/2$.
With this choice of $\theta_n$, we thus verify \eqref{eq61}.
\end{proof}

Now we can complete the proof of Theorem \ref{thm1}.

\begin{proof}[Proof of Theorem \ref{thm1}.]
Part (i) is Lemma \ref{lem2}. It remains to prove part (ii).
 Let
$R_n = \sum_{i=1}^n \1 \{ v_i \neq \eta_1(i) \} $.
Then, by Lemma \ref{lem4}, $\Exp R_n = o(n)$, which gives \eqref{nnlim}.
We obtain the limit result \eqref{gpa_to_ong} by constructing the GPA graph and ONG on a common probability space.
Indeed, given $\cX_n$ and $G_n$, one can transform the GPA graph $G_n$ into the ONG 
on the same vertex sequence by the reassignment of the endpoint with smaller index of $R_n$ edges, a transformation
that affects the degrees of at most $2 R_n$ vertices. Hence, with this coupling, for any $k \in \N$,
\[ n^{-1} \left|   N^{\mathrm{GPA}}_n (k)  -     N^{\mathrm{ONG}}_n (k)     \right| \leq 2 n^{-1}  R_n , \]
which tends to $0$ in $L^1$. Now the $L^1$ limit statement in \eqref{deglim}
yields \eqref{gpa_to_ong}.
\end{proof}

\section{Proofs for power-law attractiveness}

\subsection{Rejecting on-line nearest-neighbours}

Take $F(r) = r^{-s}$ for $s \in (0,\infty)$.
To prove Theorem \ref{thm:int2}, we consider the event $\{ v_n \neq \eta_1 (n)\}$
 that $X_n$ is joined to a point {\em other than} its nearest neighbour. First we introduce some notation
 on Voronoi cells that will also be used in analysis of the ONG in Section \ref{sec:ongproof}.
Let $\cV_n(i)$ denote the (bounded) Voronoi cell of $X_i$ with respect to $\cX_n$ in $S$, i.e.,
\begin{equation}
\label{voronoi}
 \cV_n (i) := \{ x \in S : \rho (x , X_i ) < \min \{ \rho (x , X_j) : 0 \leq j \leq n, \, j \neq i \} \} .\end{equation}
We need an elementary result showing that Voronoi cells are unlikely to be very small.
\begin{lemma}
Suppose that $\sup_{x \in S} f(x) = \lambda_1 < \infty$. Then, for any $z >0$,
\begin{equation}
\label{626b}
 \Pr [ | \cV_n (i) | < z ] \leq 2^d   \lambda_1 \delta_S^{-1} n  z ,\end{equation}
 where $\delta_S >0$ is the constant in Lemma \ref{balls}.
\end{lemma}
\begin{proof}
We follow the idea from \cite[p.\ 311]{bbbcr} (see equation (2) there).
If none of the $n$ points $X_j$ with $0 \leq j \leq n$ and $j \neq i$   lies in $B ( X_i ; r )$,
then $S \cap B (X_i ; r/2)$ is contained in $\cV_n(i)$ and hence, by Lemma \ref{balls},
$| \cV_n (i) | \geq \delta_S \omega_d (r/2)^d$. That is,
$\Pr [ | \cV_n (i) | \geq \delta_S \omega_d (r/2)^d ] \geq \Pr [ \cX_n \cap B ( X_i ; r ) = \{ X_i \} ]$.
Complementation then shows that $| \cV_n(i) | < z$ ($z >0$) implies
that at least one of $n$ points $X_j$ falls in $B ( X_i ; 2 z^{1/d} / (\omega_d \delta_S)^{1/d}  )$.
Hence, by Boole's inequality,
\[
 \Pr [ | \cV_n (i) | < z ] \leq n \Pr [ X_j \in B ( X_i ; 2 z^{1/d} / (\omega_d \delta_S)^{1/d} ) ]
\leq  2^d \delta_S^{-1}  \lambda_1 n  z ,\]
which gives \eqref{626b}.
\end{proof}

Now we can complete the proof of Theorem \ref{thm:int2}.

\begin{proof}[Proof of Theorem \ref{thm:int2}.]
Extending the notation of \eqref{nnnot}, for $\ell \in \N$ we let $\eta_\ell (n)$ be the index of the $\ell$th nearest neighbour of $X_n$ among $\cX_{n-1}$.
Again set $Z_n = \rho ( X_n, X_{\eta_1 (n)} )$ 
and $W_n = \rho ( X_n, X_{\eta_2 (n)} )$.
Then by \eqref{rule2},
\[ \frac{\Pr [ v_{n} \neq \eta_1 (n) \mid \cF_n ]}{\Pr [ v_{n} = \eta_1 (n) \mid \cF_n ]}
  \geq  \frac{\Pr [ v_{n} = \eta_2 (n) \mid \cF_n ]}{\Pr [ v_{n} = \eta_1 (n) \mid \cF_n ]}
  \geq \frac{F (W_n)}{\deg_{n-1} ( \eta_1 (n)) F (Z_n ) } .\]
  Re-arranging and using the fact that $F(r) =r^{-s}$, we obtain
  \begin{align}
  \label{e10}
   \Pr [ v_{n} \neq \eta_1 (n) \mid \cF_n ] \geq \frac{1}{1 + \deg_{n-1} ( \eta_1 (n) ) (W_n/Z_n)^s } \geq \frac{1}{2 \deg_{n-1} ( \eta_1 (n) )} \left( \frac{Z_n}{W_n} \right)^s . \end{align}
  Then \eqref{notnn} will follow from \eqref{e10} together with the following two claims:
  first, there exist constants $k_0 \in \N$ and $\theta_0 \in (0,1)$ such that
  \begin{align}
  \label{e11}
 \liminf_{n \to \infty} \Pr [ \deg_n (\eta_1 (n+1) ) \leq k_0 ] \geq 2 \theta_0 ,\end{align}
  and second, that for any $\theta >0$ there exist constants $c, C \in (0,\infty)$ such that, 
  \begin{align}
  \label{e12}
  \Pr [ Z_n \geq c n^{-1/d} ] \geq 1 - (\theta/3), ~ ~\textrm{and} ~ ~ \Pr [ W_n \leq C n^{-1/d} ] \geq 1 - (\theta/3),
  \end{align}
for all $n$ sufficiently large.
  Indeed, it follows from \eqref{e12} that
  $\Pr [ Z_n / W_n \geq c / C ] \geq 1 - (2 \theta_0 /3)$ for suitable choice of $c$ and $C$, so that, by \eqref{e10} and \eqref{e11},
  $\Pr [ v_{n} \neq \eta_1 (n) \mid \cF_n  ] \geq \frac{1}{2 k_0} (c / C)^s$
  with probability at least $\theta_0/3$ for all $n$ sufficiently large. Then, taking expectations, we obtain \eqref{notnn}.
  Thus it remains to prove the claims \eqref{e11} and \eqref{e12}.

 To verify \eqref{e11},
the idea is that there must be a large proportion of vertices with degrees bounded above by some $k_0$,
and the union of the Voronoi cells associated with these vertices will have volume bounded uniformly below in expectation,
so that $X_{n+1}$ will have such a vertex as its nearest neighbour with strictly positive probability.
We formalize this idea.

With $I_n (k) := \{ i \in \{0,\ldots, n\} : \deg_n (i) \leq k \}$,
we have $\# I_n (k) = n + 1 - N^\mathrm{GPA}_n (k+1)$.
Then taking  $k_0 =  9$, we obtain from  Lemma \ref{deglem}   that
  $\# I_n (k_0)  \geq 4n/5$ for all $n$.
Each vertex $i \in I_n (k_0)$  is associated with a Voronoi cell $\cV_n (i)$.

Let $\Lambda_n (r) = \# \{ i \in \{0,\ldots, n \} : | \cV_n (i) | \geq r / n \}$. Then
\[ \Exp [ \Lambda_n (r) ] = \sum_{i=0}^n \Pr [ | \cV_n(i) | \geq r/n ] = (n+1) \Pr [ | \cV_n (i) |  \geq r/n ] ,\]
by exchangeability. Here, by \eqref{626b},
$\Pr [ | \cV_n (i) |  \geq r/n ] \geq 1 - 2^d \lambda_1 \delta_S^{-1} r$.
Hence we can (and do)
 choose $r = r_0$ sufficiently small so that
$\Exp [ \Lambda_n (r_0) ] \geq 9n/10$, say. Then, by Markov's inequality and the fact that
$\Lambda_n (r_0) \leq 1+n$,
\[ \Pr [ \Lambda_n (r_0) \leq n/ 2 ] \leq \Pr [ n + 1 - \Lambda_n (r_0) \geq n/2 ] \leq \frac{1 + (n/10)}{n/2} \leq 1/4 ,\]
for all $n \geq 40$. So $\Pr [ \Lambda_n (r_0) \geq n/2 ] \geq 3/4$ for all $n \geq 40$.
On $\{ \Lambda_n (r_0) \geq n/2 \}$, since $\# I_n (k_0) \geq 4n/5$, there are at least
$3n/10$ vertices in $I_n(k_0)$ whose Voronoi cells all have volume at least $r_0/n$, so that
\begin{equation}
\label{e13}
 \Pr \Bigg[ \Bigg| \bigcup_{i \in I_n (k_0)} \cV_n (i) \Bigg| \geq 3r_0 /10 \Bigg] \geq 3/4 ,\end{equation}
 for all $n$ sufficiently large.
Hence
\[ \Pr \left[ \deg_n ( \eta_1 (n+1) ) \leq k_0 \right] \geq \Pr \Bigg[ X_{n+1} \in \bigcup_{i \in I_n (k_0)} \cV_n (i)
\Bigg] \geq \lambda_0 \Exp \Bigg[ \Bigg| \bigcup_{i \in I_n (k_0)} \cV_n (i) \Bigg| \Bigg] ,\]
which with \eqref{e13} gives \eqref{e11}, for $2 \theta_0 = 9 r_0 \lambda_0 / 40 >0$.

Finally, \eqref{e12} can be verified by a similar argument to Lemma \ref{nnlem}.
\end{proof}

\subsection{Stretched exponential degree estimates}
\label{sec:stretched-exponential}

Recall that $\cF_n$ denotes the $\sigma$-algebra generated by $\cX_n$ and $v_2, v_3, \ldots, v_{n-1}$
(so  the graph $G_{n-1}$ can be constructed given $\cF_n$). We also introduce the notation
  $\tilde{\cF}_{n}$ 
  for the $\sigma$-algebra generated by $\cX_n$ and $v_2, v_3, \ldots, v_{n}$
  (which includes information about $G_n$ as well).
Throughout this section we take $F(r)=r^{-s}$ for $s > d$, and   assume that
\eqref{fcon} holds.

By \eqref{rule2}, for $0 \leq i \leq n-1$,
\[ \Pr [ v_n=i \mid \cF_n ] = \frac{ \deg_{n-1}(i) \rho(X_i,X_n)^{-s}}
{\sum_{j=0}^{n-1} \deg_{n-1}(j)\rho(X_j,X_n)^{-s}}. \]
Define, for any $x \in S$,
\begin{equation}
\label{eq:zeta_def}
\zeta_{n-1} (x) := n^{-s/d} \sum_{j=0}^{n-1} \rho (X_j , x)^{-s} .\end{equation}
Then we can write
\begin{equation}
\label{stable}
\Pr [ v_n=i \mid \cF_n ] \leq \frac{\deg_{n-1}(i)\rho(X_i,X_n)^{-s}}{n^{s/d} \zeta_{n-1}(X_n)}.
\end{equation}
The next result gives an estimate for the probability that $\zeta_{n-1} (X_n)$ is small.

\begin{lemma}
\label{lem:zeta_unconditional}
There exist constants $C_0 < \infty$ and $u_0 >0$ such that,
for all $t > 0$, 
\begin{equation}
\label{eq:zeta_bound} \limsup_{n \to \infty}   \Pr [ \zeta_{n -1} (X_n) \leq t ]  
  \leq C_0 \exp \{ - u_0 t^{-d/(s-d)} \} .
\end{equation}
\end{lemma}
\begin{proof}
First, for fixed $x \in S$, we give a tail estimate for the i.i.d.\ nonnegative
random variables $\rho (X_j , x)^{-s}$ appearing in \eqref{eq:zeta_def}.
We have, for $r >0$,
\begin{align*}
\Pr [ \rho (X_j , x)^{-s} > r ] = \Pr [ X_j \in B (x ; r^{-1/s} ) ] \geq \lambda_0 \delta_S \omega_d r^{-d/s} , 
\end{align*}
using the lower bound in 
\eqref{vol0} and with $\lambda_0 = \inf_{x\in S} f(x) >0$. Hence the normalized sum $\zeta_{n-1} (x)$
stochastically dominates 
\[ \zeta_{n-1} := n^{-s/d} \sum_{j=0}^{n-1} \xi_j ,\]
where the $\xi_j$ are i.i.d.\ nonnegative random variables with
$\Pr [ \xi_j > r ] = \lambda_0 \delta_S \omega_d r^{-d/s}$.
Now, the $\xi_j$ are
in the normal domain of attraction of a positive stable law
with index $d/s \in (0,1)$, so that $\zeta_{n-1}$ converges 
in distribution as $n \to \infty$ to $\zeta$, a random variable with a positive stable law with index $d/s \in (0,1)$.
Hence, for all $x \in S$ and any $t >0$,
\[ \limsup_{n \to \infty} \Pr [ \zeta_{n -1} (x) \leq t ] \leq \lim_{n \to \infty} \Pr [ \zeta_{n-1} \leq t ] = \Pr [ \zeta \leq t ] . \]

Given that $\zeta$ is a random variable with a positive stable law with index $\alpha \in (0,1)$,
 for $p >0$ the random variable $\zeta^{-p}$ satisfies $\Exp [ \exp ( u \zeta^{-p} ) ] < \infty$
for $u \geq 0$ in a neighbourhood of zero, provided $p \leq \frac{\alpha}{1-\alpha}$: see e.g.\
the proof of Lemma 1 in \cite{bb}. Hence there exist $u_0 >0$
and $C_0 < \infty$ such that, for $p = \frac{d}{s-d} >0$, $\Exp [ \exp ( u_0 \zeta^{-p} ) ] \leq C_0$.
Thus
\[ \Pr [ \zeta \leq t ] = \Pr [ \exp ( u_0 \zeta^{-p} ) \geq \exp (  u_0 t^{-p} ) ] ,\]
and the result now follows from Markov's inequality.
\end{proof}

The next result is a conditional version of \eqref{eq:zeta_bound}, given $\cX_{n-1}$. The proof
uses
 a  concentration argument
based on
 independently `resampling' sites; a similar
 idea
will be used also in the proof of Lemma \ref{ongconc} below.
Let $X_0', X_1', \ldots$ be an independent copy of the sequence  $X_0, X_1, \ldots$.
For $0 \leq i \leq n$, let $\cX^i_n = ( X_0, \ldots, X_{i-1}, X'_i, X_{i+1}, \ldots, X_n)$, the  sites
$\cX_n$ but with the location of vertex $i$ independently resampled.
 
\begin{lemma}
\label{lem:zeta_conditional}
There exist constants $C_1 < \infty$ and $u_1 >0$ such that,
for any $t > 0$, a.s.,
\[ \limsup_{n \to \infty} \Pr [ \zeta_{n-1} (X_n) \leq t \mid \cX_{n-1} ] \leq C_1 \exp \{ - u_1 t^{-d/(s-d)} \}   .\]
\end{lemma}
\begin{proof}
We approximate the indicator function $\1_{[0,t]}$ by 
 $\chi_t^n : \RP \to [0,1]$ defined by
\[ \chi_t^n (x) := \begin{cases} 
1 & \text{ if } x \leq t \\
1 - (x-t)n^\delta & \text{ if } t \leq x \leq t + n^{-\delta} \\
0 & \text{ if } x \geq t + n^{-\delta} ,\end{cases}
\]
where $\delta >0$ is a constant to be specified later.
Then
\begin{align*}
 \Pr [ \zeta_{n-1} (X_n) \leq t \mid \cX_{n-1} ] & = \Exp [ \1_{[0,t]} ( \zeta_{n-1} ( X_n) ) \mid \cX_{n-1} ] \\
& \leq \Exp [ \chi_t^n ( \zeta_{n-1} (X_n) ) \mid \cX_{n-1} ] .\end{align*}  
Moreover, $\chi_t^n$ has the Lipschitz property
\begin{equation}
\label{eq:chi_lip}
 \chi_t^n (r) - \chi_t^n (s) \leq n^\delta (s-r)^+ .\end{equation}
We have that
\[ \Exp [ \chi_t^n ( \zeta_{n-1} (X_n) ) \mid \cX_{n-1} ] = \int_{  S} f(x) \chi_t^n ( \zeta_{n-1} (x) ) \ud x = \phi ( \cX_{n-1} ) \]
for some measurable $\phi : S^n \to [0,1]$.
To obtain a concentration result for $\phi (\cX_{n-1})$, we estimate
$\phi (\cX^i_{n-1} ) - \phi (\cX_{n-1} )$,
the change in $\phi$ on independently resampling $X_i$. 
We introduce the notation
\begin{equation}
\label{eq:zetai_def}
 \zeta^i_{n-1} (x) = \zeta_{n-1} (x) + n^{-s/d} \left( \rho (X_i' , x)^{-s} -  \rho (X_i , x)^{-s} \right) ,\end{equation}
the change in the quantity given by \eqref{eq:zeta_def} on resampling $X_i$.
Then, for $r_n >0$, 
\begin{align*}
\phi (\cX^i_{n-1} ) - \phi (\cX_{n-1} ) & \leq  \int_{B ( X_i ; r_n)} f(x) \ud x + \int_{S \setminus  B ( X_i ; r_n) } 
f(x) \left( 
\chi_t^n ( \zeta^i_{n-1} (x) ) - \chi_t^n ( \zeta_{n-1} (x) ) \right) \ud x \\
& \leq \lambda_1 \omega_d r_n^d + \int_{S \setminus  B ( X_i ; r_n)} n^\delta 
  f(x) \left( 
 \zeta_{n-1} (x) - \zeta^i_{n-1} (x)  \right)^+ \ud x ,
\end{align*}
using \eqref{eq:chi_lip}. Now, by \eqref{eq:zetai_def},
\[ \left( \zeta_{n-1} (x) - \zeta^i_{n-1} (x) \right)^+
  \leq n^{-s/d} \rho( x, X_i)^{-s} \leq n^{-s/d} r_n^{-s} ,\]
provided $x \notin B ( X_i ; r_n)$. So we obtain
\[ \phi (\cX^i_{n-1} ) - \phi (\cX_{n-1} )  \leq \lambda_1 \omega_d r_n^d + n^\delta n^{-s/d} r_n^{-s} .\]
Since $s >d$, we may choose $\delta >0$ such that $(s/d) - \delta > 1$.
Take $r_n = n^{-\nu}$
where $\nu =  \frac{(s/d) -\delta}{s+d } > 0$. Then we have that, for some constant $C< \infty$,
\[  \phi (\cX^i_{n-1} ) - \phi (\cX_{n-1} ) \leq C n^{-\frac{d ((s/d)-\delta)}{s+d}} \leq
 C n^{-\frac{d}{s+d} } .\]
 Now an appropriate version of Talagrand's inequality, Theorem 4.5 of McDiarmid \cite{mcd}, yields,
 for some $c_1>0$, for all $r >0$,
 \begin{equation}
 \label{eq:phi_tal0}
  \Pr [  | \phi (\cX_{n-1} ) - m_{n-1} | \geq r  ] \leq 4 \exp \left\{ - c_1 n^{\frac{2d}{s+d}}  r^2 \right\} ,\end{equation}
where $m_{n-1}$ is a median of $\phi (\cX_{n-1} )$. In turn, \eqref{eq:phi_tal0} implies, by Lemma 4.6 of \cite{mcd},
that $| m_{n-1} - \Exp \phi (\cX_{n-1} ) | \leq c_2 n^{-\frac{d}{s+d}}$ for some $c_2 < \infty$.
 Here 
 \[ \Exp \phi (\cX_{n-1} )  = \Exp [ \chi_t^n ( \zeta_{n-1} (X_n) )   ] \geq \Pr [ \zeta_{n-1} (X_n) \leq t   ] , \]
 which for a fixed $t>0$ is bounded below uniformly in $n$,
as can be proved using an analogous argument to the proof of Lemma \ref{lem:zeta_unconditional},
this time using the upper bound in \eqref{vol0}.
It follows that, for some $c_3 >0$,
\begin{equation}
 \label{eq:phi_tal}
 \Pr [    \phi (\cX_{n-1} ) \geq 2 \Exp \phi (\cX_{n-1} )  ] \leq 4 \exp \left\{ - c_3 n^{\frac{2d}{s+d}}   \right\}.
\end{equation}
The right-hand side of \eqref{eq:phi_tal}
 is summable in $n$, so the Borel--Cantelli lemma shows 
 \[  \Pr [ \zeta_{n-1} (X_n) \leq t \mid \cX_{n-1} ] \leq \phi (\cX_{n-1} ) \leq 2 \Exp \phi (\cX_{n-1} ) , \as ,\]
for all but finitely many $n$.       
 Here, for $t >0$,
 \[ \Exp \phi (\cX_{n-1} ) \leq  \Pr [ \zeta_{n-1} (X_n) \leq t +n^{-\delta}  ] \leq \Pr [ \zeta_{n-1} (X_n) \leq 2 t   ] \]
 for all $n$ large enough. Now the statement follows from \eqref{eq:zeta_bound}.
\end{proof}
 
Choosing $t = k^{-\gamma (s-d)/d}$ with $\gamma \in (0,1)$ in Lemma   \ref{lem:zeta_conditional}, we obtain the key estimate
\begin{equation}
\label{zeta}
\limsup_{n \to \infty}   \Pr [ \zeta_{n -1} (X_n) \leq k^{-\gamma (s-d)/d} \mid \cX_{n-1} ]
\leq C_1 \exp \{ - u_1 k^\gamma \} , \as \end{equation}
In what follows, $C_2, C_3, \ldots$ represent constants not depending on $n$ or $k$.
We have, for any $B>0$ and $t >0$,
\begin{align}
\label{split}
\Pr [ v_n=i, \, \zeta_{n-1}(X_n)> t \mid \tilde{\cF}_{n-1} ] 
  \leq \Pr [  \rho(X_i,X_n) \leq Bn^{-1/d} \mid \tilde{\cF}_{n-1} ]
\nonumber\\
   {} + \Pr [ v_n=i, \, \rho(X_i,X_n)> Bn^{-1/d}, \, \zeta_{n-1}(X_n) > t \mid \tilde{\cF}_{n-1} ].
\end{align}
The first term on the right-hand side of \eqref{split}
 is at most $C_2 B^dn^{-1}$, and the second term, 
by \eqref{stable}, 
is bounded above by 
\[ \frac{\deg_{n-1}(i)}{t n^{s/d}} \int_S f (x) \rho(X_i,x)^{-s} \1 \{ \rho(X_i,x)>Bn^{-1/d} \}   \ud x .\]
For $s>d$, the latter integral is bounded above by  
\[ C_3 \int_{Bn^{-1/d}}^{\infty}  \rho^{-s}\rho^{d-1} \ud \rho
= C_4 B^{d-s} n^{(s/d)-1} .\]
Hence we obtain  from \eqref{split} that
 \begin{equation}
 \label{indverx}
 \Pr [ v_n=i, \, \zeta_{n-1}(X_n)> t \mid \tilde{\cF}_{n-1} ]
 \leq n^{-1} \left(C_2 B^d + \frac{C_4}{t} B^{d-s} \deg_{n-1}(i) \right).
 \end{equation}

For ease of notation, let $\hier[q]{n}{k}$ be the proportion of vertices of $G_n$ with degree at least $k$, 
so that $\hier[q]{n}{k} := (n+1)^{-1} N^{\rm GPA}_n (k)$.
Then 
 the proportion of vertices of $G_n$ with degree $k$
 is equal to $\hier[q]{n}{k} - \hier[q]{n}{k+1}$, so that \eqref{indverx} yields
\begin{align*}
 & \Pr [ \deg_{n-1}(v_n) = k, \, \zeta_{n-1}(X_n)> t \mid  \tilde{\cF}_{n-1} ]
  =  \sum_{i : \deg_{n-1} (i) = k } \Pr [ v_n=i, \, \zeta_{n-1}(X_n)> t \mid \tilde{\cF}_{n-1} ] \nonumber\\
 & \qquad\qquad\qquad\qquad\qquad\qquad\qquad {} 
\leq  \left( \hier[q]{n-1}{k} - \hier[q]{n-1}{k+1} \right) \left(C_2 B^d + \frac{C_4}{t}B^{d-s}k\right).   \end{align*}
We take $t = k^{-\gamma (s-d)/d}$ for $\gamma \in (0,1)$,
and choose $B = k^{(\gamma/d) + (1/s) (1-\gamma)}$ to get
\begin{align*}
& \Pr [ \deg_{n-1}(v_n) = k, \, \zeta_{n-1}(X_n)> k^{-\gamma (s-d)/d} \mid  \tilde{\cF}_{n-1} ]
   \leq  C_5 \left( \hier[q]{n-1}{k} - \hier[q]{n-1}{k+1} \right)   k^{\gamma + (d/s) (1-\gamma)}  .   \end{align*}
Now incorporating the case where $\zeta_{n-1}(X_n)$ is small, using \eqref{zeta}, gives, a.s.,
for all $n$ sufficiently large,
\begin{align}
\label{degk}
& \Pr [ \deg_{n-1}(v_n)=k \mid \tilde{\cF}_{n-1} ]  
  \leq C_6 \re^ { - u_1 k^\gamma } + C_5 \left( \hier[q]{n-1}{k} - \hier[q]{n-1}{k+1} \right) k^{\beta} 
 ,\end{align}
 where for notational ease we have set $\beta = \gamma + (d/s) (1-\gamma)$.
 For any $k$, between times $n-1$ and $n$, the number of vertices
 of degree at least $k$ either stays the same, or increases by exactly one;
 it increases if and only if
  $\deg_{n-1}(v_n)=k-1$, so that $\deg_n (v_n) = k$.
  Thus 
  \[ \Exp [ \hier[q]{n}{k+1} \mid \tilde{\cF}_{n-1} ] - \hier[q]{n-1}{k+1} = \frac{1}{n+1} \left( n \hier[q]{n-1}{k+1} 
+ \Pr [ \deg_{n-1}(v_n)=k \mid \tilde{\cF}_{n-1} ] \right) - \hier[q]{n-1}{k+1} ,\]
  and 
  we may express \eqref{degk} as
\begin{align}
\label{degk2}
  \Exp [ \hier[q]{n}{k+1} \mid \tilde{\cF}_{n-1} ]  - \hier[q]{n-1}{k+1}  
 & \leq \frac{1}{n+1} \left( 
C_6 \re^ { - u_1 k^\gamma }  + C_5  \left( \hier[q]{n-1}{k} - \hier[q]{n-1}{k+1} \right) k^{\beta} 
-   \hier[q]{n-1}{k+1} 
  \right) \nonumber\\
 & = \frac{1}{n+1} \left( 
C_6 \re^ { - u_1 k^\gamma }  +   \hier[q]{n-1}{k} C_5 k^{\beta}  - \hier[q]{n-1}{k+1}  (1 + C_5 k^{\beta} )
  \right)
 .\end{align}
 
If we suppose that $\hier[q]{n}{k}\leq \tau_k$ for some $\tau_k$ and all $n$ sufficiently large  
(which we can, of course, always do for $\tau_k=1$) 
then \eqref{degk2} gives, for $n$ large enough,
\begin{equation}
\label{qstau}
\Exp [ \hier[q]{n}{k+1} \mid \tilde{\cF}_{n-1} ] -\hier[q]{n-1}{k+1} 
\leq \frac{1}{n+1}
\left(
C_6 \re^ { - u_1 k^\gamma }  +   \tau_k C_5 k^{\beta}  - \hier[q]{n-1}{k+1}  (1 + C_5 k^{\beta} )
\right).\end{equation} 

The final step in the proof of Theorem \ref{thm:int1} is an analysis of
 \eqref{qstau} that will enable us to iteratively improve the bound $\tau_k$. The
first part of the analysis of \eqref{qstau} will make use of
 the following stochastic approximation
result, which is related to 
  Lemma 2.6 of \cite{pemantle} and of some independent interest.
  
  \begin{lemma}
  \label{lem:stoch_approx}
  Let $(\cG_n ; n \in \ZP )$ be  a filtration.
  Let $g$ be a bounded function on $\RP$. For $n \in \ZP$, let $Y_n, r_n, \xi_n$ be $\cG_n$-measurable random variables, with $Y_n \in \RP$,
  and 
  \begin{equation}
\label{eq:stoch_approx}
Y_{n+1} - Y_n \leq \gamma_n \left( g ( Y_n) + \xi_{n+1} + r_n \right) ,\end{equation}
for constants $\gamma_n >0$.
  Suppose also that
  \begin{itemize}
  \item[(i)] $\Exp [ \xi_{n+1} \mid \cG_n ] =0$ and $\Exp [ \xi_{n+1}^2 \mid \cG_n ] \leq C$ for some constant $C<\infty$;
  \item[(ii)] $\sum_n \gamma_n = \infty$, $\sum_n \gamma_n^2 < \infty$, and $\sum_{n} \gamma_n | r_n | < \infty$ a.s.;
  \item[(iii)] $g(y) < - \delta$ for $y > y_0$ for constants $\delta >0$ and $y_0 \in \RP$.
  \end{itemize}
  Then $\limsup_{n \to \infty} Y_n \leq y_0$, a.s.
  \end{lemma}
  \begin{proof}
Summing both sides of \eqref{eq:stoch_approx} we obtain
$Y_n - Y_0 \leq M_n + A_n$ for any $n \in \ZP$, where
  \[ M_n = \sum_{k=0}^{n-1} \gamma_k \xi_{k+1} , ~~\text{and}~~ A_n = \sum_{k=0}^{n-1} \gamma_k \left( g (Y_k) + r_k \right) .\]
  Note $M_n$ is a  $\cG_n$-martingale and $A_n$ is $\cG_{n-1}$ measurable; $M_n +A_n$ is essentially
the Doob decomposition of the process whose increments are the right-hand side of  \eqref{eq:stoch_approx}.
By (i),
  \[  \Exp [ M_{n+1}^2 - M_n^2 \mid \cG_n ] = \Exp [ ( M_{n+1} - M_n )^2 \mid \cG_n ] \leq C \gamma_{n}^2 , \as, \]
  which is summable, by (ii), so the increasing process associated with $M_n$ is a.s.\ bounded. Hence $M_n \to M_\infty$
  a.s., for some finite limit $M_\infty$. Also, writing $R_n = \sum_{k=0}^{n-1} \gamma_k r_k$, we have $R_n \to R_\infty$ a.s.\
  for some finite limit $R_\infty$, by (ii).
  In particular,  for any $\eps >0$, there exists an a.s.~finite $N$ such that, 
  \[ \max_{n \geq N} \max_{m \geq 0} | M_{n+m} - M_n | \leq  \eps/4  , ~~\text{and}~~
   \max_{n \geq N} \max_{m \geq 0} | R_{n+m} - R_n | \leq  \eps/4  .\]
  Consider some $n \geq N$ for which $Y_n > y_0$. Let $\kappa_n$ be the first time
  after $n$ for which $Y_{\cdot} \leq y_0$. Then, by summing \eqref{eq:stoch_approx} again, for $m \geq 0$,
   \begin{align*}
   Y_{(n+m)\wedge \kappa_n} - Y_n & \leq M_{(n+m)\wedge \kappa_n} - M_n + R_{(n+m)\wedge \kappa_n} - R_n + \sum_{k=n}^{(n+m)\wedge \kappa_n-1} \gamma_k g (Y_k) \\
   & \leq  \frac{\eps}{2}  - \delta \sum_{k=n}^{(n+m)\wedge \kappa_n -1} \gamma_k .
   \end{align*}
In particular,  on $\{ \kappa_n = \infty \}$, letting $m \to \infty$ the left-hand side of the last display remains bounded
below by $-Y_n$ while the right-hand side tends to $-\infty$, by (ii); hence $\kappa_n < \infty$ a.s.,
and the process returns to the interval $[0, y_0]$ without 
exceeding $Y_n + \eps$.
Moreover,  
\[ Y_{n+1} - Y_n \leq \frac{\eps}{2} + \gamma_n g (Y_n) < \eps ,\]
for all $n\geq N$ large enough, since $g$ is bounded and $\gamma_n \to 0$.
 
Hence $Y_n \leq y_0$ infinitely often,
and, for all but finitely many such $n$, 
any exit from $[0,y_0]$ cannot exceed $y_0 +\eps$; but starting
from $[y_0, y_0 + \eps]$ the process returns to $[0,y_0]$
before reaching $y_0 + 2\eps$. Hence $\limsup_{n \to \infty} Y_n \leq y_0 + 2 \eps$, a.s.
Since $\eps>0$ was arbitrary, the result follows.
  \end{proof}

Now we can complete the proof of   Theorem \ref{thm:int1}.

\begin{proof}[Proof of Theorem \ref{thm:int1}.]
We   apply Lemma \ref{lem:stoch_approx} to \eqref{qstau}, with $\cG_n = \tilde \cF_n$,
$Y_n = \hier[q]{n}{k+1}$, $\gamma_n = \frac{1}{n+2}$,
$r_n = 0$, 
\[ g(y) = C_6 \re^ { - u_1 k^\gamma }  +   \tau_k C_5 k^{\beta}  - y  (1 + C_5 k^{\beta} ) ,
~~\text{and}~~ \xi_{n+1} = (n+2) \left( \hier[q]{n+1}{k+1} - \Exp [ \hier[q]{n+1}{k+1} \mid \tilde \cF_n ] \right).\]
Note that, since $N^{\rm GPA}_{n} (k)$ is $\tilde \cF_n$-measurable,
\begin{align*}
 \xi_{n+1} & = N^{\rm GPA}_{n+1} (k+1)  - \Exp [ N^{\rm GPA}_{n+1} (k+1) \mid \tilde \cF_n ] \\
& = N^{\rm GPA}_{n+1} (k+1) - N^{\rm GPA}_{n} (k+1) - \Exp [  N^{\rm GPA}_{n+1} (k+1) - N^{\rm GPA}_{n} (k+1) \mid \tilde \cF_n ] ,\end{align*}
which is uniformly bounded, since $0 \leq N^{\rm GPA}_{n+1} (k) - N^{\rm GPA}_{n} (k)  \leq 1$, a.s.
Hence the conditions of Lemma \ref{lem:stoch_approx} are satisfied for any
\[ y_0 > \frac{C_6 \re^ { - u_1 k^\gamma }  +   \tau_k C_5 k^{\beta}}{1 + C_5 k^{\beta} } ,\]
and we deduce that
 \begin{equation}
\label{tau}
\limsup_{n\to\infty}
\hier[q]{n}{k+1}
\leq            
\frac{ C_6 \re^ { - u_1 k^\gamma } + \tau_k C_5 k^{\beta}}{1+C_5 k^{\beta}}.
\end{equation} 
In particular, if $\hier[q]{n}{k} \leq \tau_k$ for all but finitely many $n$, a.s., then \eqref{tau} 
implies that 
$\hier[q]{n}{k+1} \leq \tau_{k+1}$ for all but finitely many $n$, a.s., where 
\begin{equation}
\label{eq:tau_iteration}
 \tau_{k+1} = \frac{ 2 C_6 \re^ { - u_1 k^\gamma } + \tau_k  C_5 k^{\beta}}{1+C_5 k^{\beta}} ;\end{equation}
the appearance of the factor of $2$ in \eqref{eq:tau_iteration}
 accounts for the fact that \eqref{tau} is a $\limsup$ statement, and
we want a bound for all but finitely many $n$.

Now we iterate \eqref{eq:tau_iteration}. We may rewrite \eqref{eq:tau_iteration} as
\[ \tau_{k+1} - \tau_k = \frac{1}{1+C_5 k^\beta} \left( 2 C_6 \re^{-u_1 k^\gamma} - \tau_k \right) .\]
Then, defining $\sigma_k > 0$ via $\tau_k = 2 C_6 \sigma_k \re^{-u_1 k^\gamma}$,
we obtain, after some algebra,
\[ \sigma_{k+1} - \sigma_k = \left( 1 - a_{k+1} + \frac{a_{k+1}}{1+C_5 k^\beta} \right) (1- \sigma_k) - (1-a_{k+1} ),\]
where
\[ a_{k+1} := \exp \left\{ -u_1 \left( (k+1)^\gamma - k^\gamma  \right) \right\} = 1 + \gamma u_1 k^{\gamma -1} + O (k^{\gamma -2} ),\]
as $k \to \infty$. Then,
assuming that $\beta < 1 -\gamma$, it is straightforward to check that, as $k \to \infty$,
\[ 1 - a_{k+1} + \frac{a_{k+1}}{1+C_5 k^\beta} \sim \frac{1}{C_5 k^\beta} .\]
Hence we may apply Lemma 1 of \cite{jordan2006}
to see that $\lim_{k \to \infty} \sigma_k = 1$, provided
$\beta < 1- \gamma$, i.e.,
$\gamma < \frac{s-d}{2s-d}$.
For any such $\gamma$, we thus obtain
 $\limsup_{n \to \infty} \hier[q]{n}{k} \leq 3 C_6 \re^{- u_1 k^\gamma}$, a.s.,
giving the almost sure statement in the theorem.

Then the reverse Fatou lemma yields the statement on expectations.
\end{proof}

\section{Proofs for the on-line nearest-neighbour graph}
\label{sec:ongproof}

In this section we work towards a proof of Theorem \ref{ongdeg} on the degree sequence
of the on-line nearest-neighbour graph. Our argument extends
 the 2-dimensional argument of \cite[\S 3.1]{bbbcr}, who considered the uniform distribution on the square.

Recall the definition of the Voronoi cell $\cV_n(i)$ from \eqref{voronoi}.
Then
\begin{equation}
\label{nesting}
\cV_{n+1} (i) = \cV_n (i) \cap \{ x \in S : \rho (x , X_i) < \rho (x , X_{n+1} ) \} \subseteq \cV_n (i) .\end{equation}

A key fact is provided by the following lemma, which will be used to show that the volume of a Voronoi cell associated
with a vertex in the ONG
shrinks, on average, by a positive fraction whenever a new vertex lands in the cell.

\begin{lemma}
\label{shrink}
Let $R \subseteq S$ be convex, and let $X$ be a random point in $S$ distributed according
to the  probability density $f$ satisfying \eqref{fcon}.
For $x_0 \in R$, let $R' = \{ x \in R : \rho (x, x_0 ) < \rho (x, X)\}$. Then there exists $\delta >0$ not depending on $R$
or $x_0$
such that
\[ \Exp [ | R' |  \mid X \in R ] \leq (1- \delta) \Exp [ | R |] .\]
\end{lemma}
\begin{proof}
Without loss of generality, suppose that $x_0 = 0 \in R$.
Partition $R$ according to the $2^d$ Cartesian orthants as $R_1, \ldots, R_{2^d}$. In any orthant $j$, any
 two points
$x, y \in R_j$ have the same signs in corresponding coordinates, so $\| x - y \| \leq \| x+ y \|$, and hence
$(x+y)/2$ is closer to $x$ (and to $y$) than to $0$. Thus, given $X \in R_j$,
 any point $x$ of
$R_j'' := \{  (X+y)/2 : y \in R_j \}$ has $\| x - X \| \leq \| x - 0 \|$, and, by convexity,
$R_j'' \subseteq R_j$. Hence, given $X \in R_j$,
$ R' \subseteq   R \setminus   R_j''$.
By construction, $  R_j''$ is a translate of $R_j$ scaled by a factor of $1/2$, so
\begin{align*}
 \Exp [  | R' | \mid X \in R ] & \leq   | R |     - \sum_{j=1}^{2^d} 2^{-d} | R_j | \Pr [ X \in R_j \mid X \in R ]  \\
 & \leq |R| - 2^{-d} (\lambda_0/\lambda_1) |R|^{-1} \sum_{j=1}^{2^d} |R_j|^2 ,
\end{align*}
where $\lambda_0 = \inf_{x \in S} f(x) >0$ and $\lambda_1 = \sup_{x \in S} f(x) < \infty$, by \eqref{fcon}.
Now, by Jensen's inequality, $\sum_{j=1}^{2^d} | R_j|^2 \geq 2^{-d} ( \sum_{j=1}^{2^d} | R_j | )^2 = 2^{-d} | R|^2$,
and the claimed result follows with $\delta = 2^{-2d} \lambda_0 /\lambda_1$.
\end{proof}

Next we give bounds on expectations for $N_n^\mathrm{ONG} (k)$.

\begin{lemma}
\label{expdeg}
Let $d \in \N$. Suppose that \eqref{fcon} holds.
Then there exist finite positive constants $A,A',C,C'$ such that, for all $k \in \N$,
\begin{equation}
\label{eq:two_sided_bounds}
 A' \re^{-C'k} \leq \liminf_{n \to \infty} n^{-1} \Exp [ N^\mathrm{ONG}_n (k) ] \leq \limsup_{n \to \infty} n^{-1} \Exp [ N^\mathrm{ONG}_n (k) ]
 \leq A \re^{-Ck} .\end{equation}
Moreover,
\begin{equation}
\label{eq:upper_bound_for_mu}
\limsup_{k \to \infty} \left(
- k^{-1} \log \left( \liminf_{n \to \infty} n^{-1} \Exp [ N^\mathrm{ONG}_n (k) ] \right) \right) \leq 1 ,
\end{equation}
and, in the case where $f$ is the uniform density on $S$,
\begin{equation}
\label{eq:lower_bound_for_mu}
\liminf_{k \to \infty} \left(
- k^{-1} \log \left( \limsup_{n \to \infty} n^{-1} \Exp [ N^\mathrm{ONG}_n (k) ] \right) \right) \geq 
\frac{1}{2} \log \left( 1 + (2^{2d} -1)^{-1} \right) .
\end{equation}
\end{lemma}
\begin{proof}
First we prove the upper bound in \eqref{eq:two_sided_bounds}, using an argument based in part on \cite[\S 3.1]{bbbcr}.
By \eqref{fcon},  $\inf_{x \in S} f(x) = \lambda_0 >0$ and $\sup_{x \in S} f(x) = \lambda_1 < \infty$.
Fix $i \in \ZP$. Let $t_0 = i$ and for $j \in \N$ define recursively
$t_j = \min \{ t > t_{j-1} : X_t \in \cV_{t-1} (i) \}$,
so that $t_1, t_2, \ldots$ are the times at which edges to $X_i$ are created.
Following \cite[p.\ 311]{bbbcr}, let $\cW_j = \cV_{t_j} (i)$.

Observe that if $i$ has degree greater than $k$ in the ONG on $(X_0, \ldots, X_n)$, $n \geq i$,
then necessarily $t_k \leq n$, and so also $| \cV_n (i) | \leq | \cV_{t_k} (i) | = | \cW_k |$, by \eqref{nesting}.
Hence, for any $z>0$,
\begin{equation}
\label{626} \Pr [ \deg_n (i) > k ] \leq \Pr [ | \cW_k | \geq z ] + \Pr [ | \cV_n (i)| \leq z ] .\end{equation}
We bound each of the probabilities on the right-hand side of \eqref{626} in turn, and then optimize the choice of $z$.

By definition,
$X_{t_j}$ is distributed according to the density $f$, conditioned to fall in the convex set $\cV_{t_{j}-1} (i) \subseteq \cV_{t_{j-1}} (i) \subseteq S$. 
Hence Lemma \ref{shrink} 
shows that
$\Exp [ | \cW_{j} | ] \leq (1-\delta) \Exp [ | \cW_{j-1}| ]$, 
where $\delta \in (0,1)$ depends only on $d$ and $\lambda_0/\lambda_1$, and
\[ \Exp [ | \cW_j | ] \leq (1-\delta )^j \Exp [ | \cV_i (i) | ] = \frac{1}{i+1} (1-\delta )^j ,\]
since the vector $(|\cV_i(0)|, \ldots, |\cV_i (i)|)$ is exchangeable and its components sum to $1$, so $\Exp [ | \cV_i (j) | ] = \frac{1}{i+1}$.
Markov's inequality   implies that, for any $z >0$,
\begin{equation}
\label{626a}
 \Pr [ | \cW_j | \geq  z ] \leq \frac{1}{z} \frac{1}{i+1} (1-\delta )^j .\end{equation}

The final term in \eqref{626} is bounded above by \eqref{626b}.
Combining \eqref{626} with \eqref{626a} and \eqref{626b}, we obtain, for any $z>0$,
\[
 \Pr [ \deg_n (i) > k ] \leq \frac{1}{z} \frac{1}{i+1} (1-\delta)^k + C n z  ,\]
where  $C < \infty$ depends only on $d$, $S$,  and $\lambda_1$.
The optimal bound is obtained on taking $z = (1-\delta)^{k/2} / \sqrt{ C n (i+1)}$, and we conclude
\begin{equation}
\label{d1} \Pr [ \deg_n (i) > k ] \leq 2 (1-\delta)^{k/2} \sqrt{ \frac{C n}{i+1} } .\end{equation}
The upper bound in \eqref{eq:two_sided_bounds}
follows from \eqref{d1}, since
\[ \Exp [ N_n^\mathrm{ONG}(k) ] = \sum_{i=0}^n \Pr [ \deg_n (i) \geq k ] \leq C' n (1-\delta)^{k/2} ,\]
for some $C' < \infty$ not depending on $k$ or $n$. 
The statement \eqref{eq:lower_bound_for_mu} also follows, since when $\lambda_0 = \lambda_1$, we have from
the proof of Lemma \ref{shrink} that we may take $\delta = 2^{-2d}$.

To prove the lower bound in \eqref{eq:two_sided_bounds} as well as \eqref{eq:upper_bound_for_mu}, we use a similar idea
to that briefly outlined for the analogous argument
in \cite[p.\ 311]{bbbcr}, but filling in the details takes some work,
and we must be more careful with our estimates to obtain the quantitative bound \eqref{eq:upper_bound_for_mu}.
First note that, for $j >i$,
the (unconditional) probability that $X_j$ is joined to $X_i$ is
$\Pr [ \eta_1 (j) = i ] = \Pr [ X_j \in \cV_{j-1} (i) ] = 1/j$. Write $ d_n (i) := \Exp [ \deg_n (i) ]$. Then, for $i \in \N$,
\[ d_n (i) = 1 + \sum_{j=i+1}^n \Pr [ \eta_1 (j) = i ] \geq \sum_{j=i}^n \frac{1}{j} \geq \int_i^n \frac{1}{y} \ud y = \log ( n / i)  .\]
Let $\theta >1$.
For $k \in \ZP$, let $H^\theta_{n,k} := \N \cap [ 1, \re^{-\theta k} n ]$.
Then for any $i \in H^\theta_{n,k}$, $d_n (i) \geq \log (n/i) \geq \theta k$. It follows that
\begin{equation}
\label{lower}
 \Exp [ N_n^\mathrm{ONG}(k) ] \geq \sum_{i \in   H^\theta_{n,k}}  \Pr [ \deg_n (i) \geq k ] \geq \sum_{i \in   H^\theta_{n,k}} \Pr [ \deg_n (i) \geq \theta^{-1} d_n (i)  ] .\end{equation}

 Let $w \in (1,\infty)$, to be specified later.
Then $w > 1 > 1/\theta$, and
\begin{align}
\label{H1}
w d_n (i) \Pr [ \deg_n (i) \geq \theta^{-1} d_n (i) ] 
& \geq 
\Exp [ \deg_n (i) \1 \{ \deg_n (i) \geq \theta^{-1} d_n (i) \} ] \nonumber\\
&{} \qquad \qquad {} - \Exp [ \deg_n (i) \1 \{ \deg_n (i) > w d_n (i) \} ] \nonumber\\
& \geq \left(1 - \theta^{-1} \right) d_n (i) - \Exp [ \deg_n (i) \1 \{ \deg_n (i) > w d_n (i) \} ] ,\end{align}
using the fact that $\Exp [ X \1 \{ X \geq x \} ] \geq \Exp [ X ] -x$ for any $x \geq 0$ and any nonnegative random variable $X$.
By the Cauchy--Schwarz inequality, the final term in \eqref{H1} satisfies
\begin{equation}
\label{H2}
  \Exp [ \deg_n (i) \1 \{ \deg_n (i) > w d_n (i) \} ] \leq \left( \Exp [   \deg_n (i) ^2 ] \Pr [ \deg_n (i) > w d_n (i)  ] \right)^{1/2} .\end{equation}
We claim that, given $\theta >1$, there exists $w = w(\theta) \in (1,\infty)$ such that
\begin{equation}
\label{Hbound}
  \sup_{ i \in H^\theta_{n,k} } \left( \Exp [   \deg_n (i)  ^2 ] \Pr [ \deg_n (i) > w d_n (i)  ] \right)^{1/2}
 \leq \re^{-\theta k}, \text{ for all } n \in \N \text{ and all } k \in \N.
\end{equation}
Given \eqref{Hbound}, which we verify at the end of this proof,
we
 obtain from \eqref{H1}, \eqref{H2}, and \eqref{Hbound}
that, for any $n \in \N$ and any $k \in \N$,
\begin{align}
\label{eq:theta_bound}
 w \inf_{i \in H^\theta_{n,k} }  \Pr [ \deg_n (i) \geq \theta^{-1} d_n (i) ] & \geq 
\left(1 - \theta^{-1} \right)  -  \re^{- \theta k} \sup_{ i \in H^\theta_{n,k}} \frac{1}{d_n (i)} \nonumber\\
& \geq \left(1 - \theta^{-1} \right) -  \frac{\re^{- \theta k}}{\theta k}  ,\end{align}
using the fact that $d_n (i) \geq \theta k$ for $i \in H^\theta_{n,k}$.
To prove the lower bound in \eqref{eq:two_sided_bounds}, it is enough to fix $\theta =2$. Then \eqref{eq:theta_bound}
becomes, for any $n \in \N$ and any $k \in \N$,
\[  w \inf_{i \in H^2_{n,k} }  \Pr [ \deg_n (i) \geq \tfrac{1}{2} d_n (i) ] \geq \frac{1}{2} \left(1 - \re^{-2} \right)
 \geq \frac{3}{8}, \]
say, where $w = w(2)$ is constant. Hence from \eqref{lower} we obtain, for all $n \in \N$ and all $k \in \N$,
\[ \Exp [ N^\mathrm{ONG}_n (k) ] \geq w^{-1}
\sum_{i \in   H^2_{n,k}}  \frac{3}{8} 
\geq \frac{3}{8w} \left( \re^{-2k} n -1 \right), \]
 which gives $\liminf_{n \to \infty} n^{-1} \Exp [ N_n^\mathrm{ONG}(k) ] \geq \tfrac{3}{8w} \re^{-2k}$.

To prove \eqref{eq:upper_bound_for_mu}, we adapt the preceding argument. For any $\theta >1$, there exists $k_0 \in \N$
such that, for all $k \geq k_0$,
the final expression on the right-hand side of \eqref{eq:theta_bound}
exceeds $\frac{1-\theta^{-1}}{2} >0$, say. Then, similarly to before, we obtain, for all $k \geq k_0$ and $n \in \N$,
\[   \Exp [ N^\mathrm{ONG}_n (k) ] \geq 
w^{-1} \sum_{i \in   H^\theta_{n,k}}  \left( \frac{1-\theta^{-1}}{2} \right) 
\geq w^{-1} \left( \frac{1-\theta^{-1}}{2} \right)   \left( \re^{- \theta k} n -1 \right) .\]
First letting $n \to \infty$ and then $k \to \infty$, it follows that
\[ \limsup_{k \to \infty} \left(
- k^{-1} \log \left( \liminf_{n \to \infty} n^{-1} \Exp [ N^\mathrm{ONG}_n (k) ] \right) \right) \leq \theta .\]
Since $\theta >1$ was arbitrary, \eqref{eq:upper_bound_for_mu} follows.

It remains to establish the claim \eqref{Hbound}. To this end, an application of \eqref{d1} shows that,
for constants $C_1, C_2 <\infty$ and $c>0$, for all $n \in \N$ and $1 \leq i \leq n$,
\[ \Exp [ \deg_n (i)^2 ] = \sum_{k=0}^\infty \Pr [ \deg_n (i) > \sqrt{k} ]
\leq C_1 \sqrt{ \frac{n}{i} }
\sum_{k=0}^\infty \re^{-c \sqrt{k}}
\leq C_2 \sqrt{ \frac{n}{i} } .\]
Another application of \eqref{d1} shows that, for some constant $C_3 < \infty$, for any $w >0$,
\[ \Pr [ \deg_n (i) > w d_n (i) ] \leq C_3 \sqrt{ \frac{n}{i} } \re^{-cw \log (n/i) } = C_3 \left( \frac{n}{i} \right)^{(1/2)-cw} .\]
Hence we obtain, for all $1 \leq i \leq n$,
\[
  \left( \Exp [   \deg_n (i)  ^2 ] \Pr [ \deg_n (i) > w d_n (i)  ] \right)^{1/2}  \leq C_4
\left( \frac{i}{n} \right)^{(cw- 1)/2 }   ,\]
where $C_4 < \infty$ is constant. 
Taking $w > 3/c$, we have, for any $i \in H^\theta_{n,k}$,
\[ C_4
\left( \frac{i}{n} \right)^{(cw- 1)/2 } 
\leq C_4 \re^{-\theta k} \re^{-(cw-3) \theta k /2 } ,\]
since $i/n \leq \re^{- \theta k}$ for $i \in H^\theta_{n,k}$.
In particular, for all $k \in \N$, we can choose $w$ (depending on $c$, $C_4$ and $\theta$)
such that $C_4 \re^{-(cw-3) \theta k /2 } \leq C_4 \re^{-(cw-3) \theta /2 } \leq 1$.
This verifies \eqref{Hbound}.
  \end{proof}

Next we have a concentration result for $N_n^\mathrm{ONG} (k)$.

\begin{lemma}
\label{ongconc}
Let $d \in \N$. Suppose that \eqref{fcon} holds. Then
\begin{equation}
\label{concentration}
 \limsup_{n \to \infty} n^{-1} \sup_{k \in \N} | N^\mathrm{ONG}_n(k) - \Exp [ N^\mathrm{ONG}_n (k) ] | =0 , \as \end{equation}
\end{lemma}
\begin{proof}
We use a  concentration argument based on a modification of the Azuma--Hoeffding inequality, which uses
the resampling idea described before Lemma \ref{lem:zeta_conditional}. Recall that $\cX^i_n$ denotes 
$\cX_n$ but with the location of vertex $i$ independently resampled.
Let  
  $\cG_i = \sigma (X_0, X_1, \ldots, X_i)$; then $N^\mathrm{ONG}_n(k) = \psi_{n,k} (\cX_n)$ a.s.\ for some
measurable  $\psi_{n,k} : S^{n+1} \to \ZP$ and $N^\mathrm{ONG}_n(k)$ is $\cG_n$-measurable.
Fix $k \in \N$, and write
\[ D_{n,i} = \Exp [ N^\mathrm{ONG}_n(k) \mid \cG_{i} ] - \Exp [ N^\mathrm{ONG}_n(k) \mid \cG_{i-1} ] = \Exp [ N^\mathrm{ONG}_n (k) - N^\mathrm{ONG}_{n,i} (k) \mid \cG_i ] ,\]
where $N^\mathrm{ONG}_{n,i} (k) = \psi_{n,k} ( \cX^i_n )$.
In   words, $-D_{n,i}$ is the expected change in $N^\mathrm{ONG}_n (k)$ (conditional on $\cG_i$) on resampling the location of the $i$th point $X_i$.
Then $D_{n,i}$, $1 \leq i \leq n$ is a martingale difference sequence with
$  \sum_{i=1}^n D_{n,i} = N^\mathrm{ONG}_n (k) - \Exp [ N^\mathrm{ONG}_n (k) ]$.

We bound $| D_{n,i} |$ in terms of $\deg_n ( i)$ and $\deg_n^i (i)$,   the degree of vertex $i$ in the ONG
on $\cX_n$ and $\cX_n^i$ respectively.
On replacement of $X_i$ by $X_i'$,
the degree of vertex $i$ may change, leading to a change of $\pm 1$ in $N^\mathrm{ONG}_{n,i} (k)$ compared to $N^\mathrm{ONG}_n (k)$.
The degrees of at most $\deg_n ( i) -1$ other vertices increase
(namely those vertices that gain incoming edges that were previously connected to $X_i$),
while the degrees of at most $\deg_n^i (i) -1$ vertices decrease (namely those vertices
that lose incoming edges re-assigned to $X_i'$).
 
Hence
$| D_{n,i } | \leq   \deg_n (i) + \deg_n^i (i)$.
 Now, for any $r >0$,
\[ \Pr [ | D_{n,i } | > r ] \leq \Pr [ \deg_n (i) > r/2] + \Pr [ \deg_n^i (i) > r/2] = 2 \Pr [ \deg_n (i) > r/2 ] ,\]
since $\deg_n^i (i)$ and $\deg_n (i)$ are identically distributed. Hence, by \eqref{d1},
$\Pr [ | D_{n,i} | > D \log n ] = O ( n^{-5})$, 
uniformly in $i$, choosing $D \in (0,\infty)$ sufficiently large; note that this bound is also uniform in $k$.
By a modification of the Azuma--Hoeffding inequality due to Chalker {\em et al.} \cite[Lemma 1]{chalker}, it follows that
\[ \Pr [ | N^\mathrm{ONG}_n(k) - \Exp [ N^\mathrm{ONG}_n (k) ] | > r ]
\leq \left( 1 + \frac{4n}{r} \right) n^{-4} + 2 \exp \left \{ - \frac{r^2}{ 32 D^2 n (\log n)^2 } \right \} ,\]
for any $r >0$.  Taking $r = n^{3/4}$, say, shows that
$\Pr [ | N^\mathrm{ONG}_n(k) - \Exp [ N^\mathrm{ONG}_n (k) ] | > n^{3/4} ] = O ( n^{-3} )$, uniformly in $k \in \{1,\ldots,n\}$,
while for $k >n$, $N^\mathrm{ONG}_n(k) = 0$ a.s. Hence
\[ \sum_{n=1}^\infty \sum_{k \in \N} \Pr [ | N^\mathrm{ONG}_n(k) - \Exp [ N^\mathrm{ONG}_n (k) ] | > n^{3/4} ]  \leq C
\sum_{n=1}^\infty n^{-2} < \infty .\]
The Borel--Cantelli lemma now yields \eqref{concentration}.
\end{proof}

Now we can complete the proof of Theorem \ref{ongdeg}.

\begin{proof}[Proof of Theorem \ref{ongdeg}.]
First we show that,
for any $k \in \N$,
\eqref{deglim} holds.
Penrose \cite[\S 3.4]{mdp} showed that functionals such as counts of vertices of a given degree in the ONG
satisfy {\em stabilization} (a form of local dependence). Stabilization of the form demonstrated in \cite{mdp}
guarantees a law of large numbers of the form $n^{-1} N^\mathrm{ONG}_n (k) \to \Pr [ \xi ( 0, U ; \cH )  \geq k]$
as $n \to \infty$, with {\em convergence in probability}: concretely, one may apply results of Penrose and Yukich
\cite{py2} or Penrose \cite{plln}. The fact that $n^{-1} \Exp [ N^\mathrm{ONG}_n (k) ] \to \Pr [ \xi ( 0, U ; \cH )  \geq k]$
then follows from the bounded convergence theorem. Lemma \ref{ongconc} now shows that convergence in probability
can be replaced by almost sure convergence, and the $L^1$ convergence follows from the bounded convergence theorem again.
 Thus \eqref{deglim} holds.

Then, applying Lemma \ref{expdeg} with \eqref{deglim},   \eqref{degbounds} 
follows from
\eqref{eq:two_sided_bounds}. Given \eqref{deglim}, the upper bound in 
\eqref{eq:rho_lower_bound}
follows from \eqref{eq:upper_bound_for_mu}. Similarly, the 
lower bound in \eqref{eq:rho_lower_bound} follows from 
\eqref{eq:lower_bound_for_mu}, noting that the limit $\rho_k$ is independent of the choice of $f$.

It is easy to see that $\rho_k$ is nonincreasing with $\rho_1 =1$. Since $\sum_{k\in \N} N^\mathrm{ONG}_n (k) = 2 n$, twice the number of edges
in the ONG, dividing both sides of this last equality by $n$ and letting $n \to \infty$ we must have $\sum_{k \in \N} \rho_k = 2$; hence also
$\lim_{k \to \infty} \rho_k = 0$.
For the final statement of the theorem,
we have from \eqref{d1} that for any $k >0$,
\[ \Pr \Big[  \max_{0 \leq i \leq n} \deg_n (i)  > k \Big] \leq  (n+1) \max_{0 \leq i \leq n} \Pr [    \deg_n (i)  > k ] \leq C n^{3/2} \re^{-c k} ,\]
for some absolute constants $c, C \in (0,\infty)$. Taking $k = D \log n$, 
we can choose $D \in (0,\infty)$ for which this last bound is $O( n^{-2} )$, say;
the Borel--Cantelli lemma then gives \eqref{ongmax}.
\end{proof}

\end{document}